\documentclass[]{amsart}
\usepackage{amscd, bbm, amsaddr}
\usepackage{amssymb,mathabx}
\usepackage{mathrsfs,wasysym, tikz-cd, extpfeil}
\usetikzlibrary{matrix, arrows}
\usepackage{amsmath, enumitem}
\usepackage[margin=2cm]{geometry}

\usepackage{soul}
\usepackage{color}
\usepackage{amscd}
\usepackage[pdftex,colorlinks,hypertexnames=false,linkcolor=black,urlcolor=black,citecolor=black]{hyperref}

\makeatletter
\@addtoreset{equation}{section}

\usepackage{setspace}
\numberwithin{equation}{section}

\newtheorem{Theorem}{Theorem}[section]

\newtheorem*{Theorem*}{Theorem}
\newtheorem*{Corollary*}{Corollary}

\newtheorem{Lemma}[Theorem]{Lemma}
\newtheorem{Proposition}[Theorem]{Proposition}
\newtheorem{Corollary}[Theorem]{Corollary}

\theoremstyle{definition}
\newtheorem{Definition}[Theorem]{Definition}

\theoremstyle{remark}
\newtheorem{Remark}[Theorem]{Remark}
\newtheorem*{Remark*}{Remark}

\newtheorem{Example}[Theorem]{Example}

\newbox\squ  
\setbox\squ=\hbox{\vrule width.6pt
   \vbox{\hrule height.6pt width.4em\kern1ex
         \hrule height.6pt}%
   \vrule width.6pt}

\newcommand{\C}{\mathbb{C}}
\newcommand{\Nbb}{\mathbb{N}}
\newcommand{\Z}{\mathbb{Z}}

\renewcommand{\Z}{\mathbb{Z}}

\newcommand{\g}{\mathfrak{g}}
\newcommand{\gl}{\mathfrak{gl}}
\newcommand{\h}{\mathfrak{h}}

\renewcommand{\sl}{\mathfrak{sl}}
\renewcommand{\so}{\mathfrak{so}}
\newcommand{\z}{\mathfrak{z}}
\newcommand{\m}{\mathfrak{m}}
\newcommand{\n}{\mathfrak{n}}

\newcommand{\Q}{\mathcal{Q}}
\newcommand{\G}{\mathcal{G}}
\newcommand{\N}{\mathcal{N}}
\renewcommand{\O}{\mathcal{O}}

\newcommand{\A}{\mathcal{A}}

\newcommand{\Ss}{\mathcal{S}}

\newcommand{\ad}{\operatorname{ad}}

\newcommand{\Lie}{\operatorname{Lie}}

\newcommand{\reg}{{\operatorname{reg}}}
\newcommand{\Hom}{\operatorname{Hom}}
\newcommand{\gr}{\operatorname{gr}}
\newcommand{\graded}{\operatorname{\tt{gr}}}
\newcommand{\Spec}{\operatorname{Spec}}

\newcommand{\Aut}{\operatorname{Aut}}
\newcommand{\PAut}{\operatorname{PAut}}

\newcommand{\Der}{\operatorname{Der}}

\newcommand{\Ann}{\operatorname{Ann}}

\newcommand{\GL}{\operatorname{GL}}
\newcommand{\SL}{\operatorname{SL}}
\newcommand{\SO}{\operatorname{SO}}
\newcommand{\filt}{{\operatorname{filt}}}

\newcommand{\Qnt}{\operatorname{Q}}
\newcommand{\PD}{\operatorname{PD}}

\newcommand{\Forg}{\operatorname{\tt{F}}}

\newcommand{\Sets}{\operatorname{\bf{Sets}}}
\newcommand{\GrAlg}{\operatorname{\bf{G}}}
\newcommand{\FAlg}{\operatorname{\bf{F}}}
\newcommand{\SFAlg}{\operatorname{\bf{SF}}}

\newcommand{\id}{\operatorname{id}}

\newcommand{\btau}{\overline{\tau}}

\newcommand{\cD}{\overset{\circ}{D}}

\newcommand{\tX}{\widetilde{X}}

\newcommand{\isoto}{\overset{\sim}{\longrightarrow}}
\newcommand{\onto}{\twoheadrightarrow}
\newcommand{\into}{\hookrightarrow}
\DeclareRobustCommand\longtwoheadrightarrow
     {\relbar\joinrel\twoheadrightarrow}

\title[Universal quantizations of nilpotent Slodowy slices]{\boldmath Universal filtered quantizations \\ of nilpotent Slodowy slices}

\author{F. Ambrosio, G. Carnovale, F. Esposito and L. Topley}

\begin{document}

\maketitle

\begin{abstract}
Every conic symplectic singularity  admits a universal Poisson deformation and a universal filtered quantization, thanks to the work of Losev and Namikawa.
We begin this paper by showing that every such variety admits a universal equivariant Poisson deformation and a universal equivariant quantization with respect to a reductive group acting on it by $\C^\times$-equivariant Poisson automorphisms. 

We go on to study these definitions in the context of nilpotent Slodowy slices.
First we give a complete description of the cases in which the finite $W$-algebra is a universal filtered quantization of the slice, building on the work of Lehn--Namikawa--Sorger.
This leads to a near-complete classification of the filtered quantizations of nilpotent Slodowy slices.

The subregular slices in non-simply-laced Lie algebras are especially interesting: with some minor restrictions on Dynkin type we prove that the finite $W$-algebra is a universal equivariant quantization with respect to the Dynkin automorphisms coming from the unfolding of the Dynkin diagram. This can be seen as a non-commutative analogue of Slodowy's theorem. Finally we apply this result to give a presentation of the subregular finite $W$-algebra in type {\sf B} as a quotient of a shifted Yangian. 
\end{abstract}

\tableofcontents

\section{Introduction}

The finite subgroups of $\SL_2(\C)$ are classified by the simply-laced Dynkin diagrams. If $\Delta$ is such a diagram corresponding to a group $\Gamma$ then the quotient singularity $\C^2 /\Gamma$ is said to have type $\Delta$. It was proven by Artin that these varieties give an exhaustive list of rational isolated surface singularities up to analytic isomorphism \cite{Art}.

The classical theorem of Brieskorn \cite{Br}, conjectured by Grothendieck, states that if $\g$ is a complex simple Lie algebra with simply-laced Dynkin diagram $\Delta$ then the transverse slice  to the subregular orbit is the $\C^\times$-semi-universal deformation of the singularity of type $\Delta$. This remarkable theorem was extended to the non-simply-laced types by Slodowy \cite{Slo}. Let $\Delta_0$ be such a diagram, and let $\Delta$ be simply-laced and $\Gamma_0 \subseteq \Aut(\Delta)$ be uniquely determined by the requirement that $\Delta_0$ is obtained by folding $\Delta$ under $\Gamma_0$ (Cf. \cite[\textsection 13]{Ca}, for example). Then the subregular slice in a Lie algebra of type $\Delta_0$ is the $\C^\times$-semi-universal $\Gamma_0$-deformation of a singularity of type $\Delta$.

Lehn, Namikawa and Sorger have generalised this classical story to arbitrary nilpotent orbits, taking the focus away from the subregular case \cite{LNS}. In general the nilpotent part of a Slodowy slice is not an isolated surface singularity and so there is no versal theory for deformations. It turns out that the correct generalisation is given by realising the Slodowy slice as a Poisson variety via Hamiltonian reduction, following \cite{GG}. Since the nilpotent part of the slice is a conic symplectic singularity, results of Namikawa \cite{Na1, Na2} show that there is a Poisson deformation which is universal, in the sense that every other deformation is obtained by a unique base change. The main result of \cite{LNS} gives a necessary and sufficient condition for the Slodowy slice to be a universal Poisson deformation of its nilpotent part.

Let $G$ be a simple, simply connected, complex algebraic group and $\g = \Lie(G)$ its Lie algebra, with nilpotent cone $\N(\g)$. Choose $e\in \N(\g)$ and identify $e$ with $\chi \in \g^*$ via the Killing isomorphism $\kappa : \g \isoto \g^*$.
Set $\N(\g^*) = \kappa(\N(\g))$.
Denote the adjoint orbit of $e$ by $\O$ and write $\Ss_\chi$ for the Slodowy slice to $\chi$ in $\g^*$, which is transverse to coadjoint orbits. Consider the Springer resolution $\pi : \widetilde\N \to \kappa(\N(\g))$.
Then Theorems~1.2 and 1.3 of \cite{LNS} state that the following are equivalent:
\begin{enumerate}[label=(\arabic*)]
\item $\varphi \colon \Ss_\chi \to \h^*/W$ is a universal Poisson deformation of the central fibre $\varphi^{-1}(0)$;
\item the restriction map $H^2(\widetilde \N, \mathbb{Q}) \to H^2(\pi^{-1}(\chi), \mathbb{Q})$ is an isomorphism;
\item $\O$ does not occur in the Table 1.
\end{enumerate}
\begin{center}

\begin{center}
{\sc Table 1: cases in which $\Ss_\chi \to \h^*/W$ is not a universal Poisson deformation.}
\end{center}\vspace{2pt}
\begin{tabular}{|c|c|c|c|c|c|c|}
\hline
Type of $\g$ & Any & {\sf BCFG} & {\sf C} & \sf{G} 

\\
\hline
Type of $\O$ & Regular & Subregular & Two Jordan blocks & dimension 8 \\ 
\hline
\end{tabular}
\end{center}\vspace{8pt}

The objective of this paper is to prove non-commutative analogues of many of the results described above: we will classify the non-commutative filtered quantizations of every nilpotent Slodowy slice whose orbit does not appear in Table~1, and we prove non-commutative analogues of both Brieskorn's and Slodowy's theorems.

Our point of departure is the work of Namikawa \cite{Na1, Na2} and Losev \cite{Lo} on deformations and quantizations of conic symplectic singularities. Namikawa has shown that every such variety admits a universal Poisson deformation, whilst Losev demonstrated that Namikawa's deformation can be quantized, and that the quantization enjoys a universal property similar to its semi-classical limit. 


In order to compare universal Poisson deformations with universal quantizations, we begin the paper working in a general setting.
We fix $n\in \Nbb$ and a commutative positively graded connected Poisson algebra $A$  with bracket in degree $-n$.
The functor $\PD_A$ of graded Poisson deformations has been well-studied by many authors, and we recap some of the main features in Section \ref{ss:poissondefsection}. The functor $\Qnt_A$ of filtered quantizations of Poisson deformations of $A$ is defined analogously, although it appears to be new to the literature; we state some of its basic properties in Section \ref{ss:quantizationdsection}.
One of the key definitions in this paper is the functor of deformations (or quantizations) with fixed symmetries.
Suppose that $\Gamma$ is a reductive group of graded Poisson automorphisms of $A$.
We define actions on deformations and quantizations of $A$ compatible with the one on the central fibre, namely graded Poisson $\Gamma$-deformations and filtered $\Gamma$-quantizations; see Section~\ref{ss_autdef} and \ref{ss_autqnt}.
Our main result in this setting is the following.


\begin{Theorem} 
\label{T:firstmain}
Suppose that $X$ is a conic symplectic singularity and $\Gamma$ is a reductive group of $\C^\times$-equivariant Poisson automorphisms of $A = \C[X]$. 
\begin{enumerate}[label=(\arabic*)]
\item The functor $\PD_{A, \Gamma}$ of graded Poisson $\Gamma$-deformations of $A$, resp. the functor $\Qnt_{A, \Gamma}$ of filtered $\Gamma$-quantizations of $A$, admit universal elements.
\item A filtered $\Gamma$-quantization is a universal element of $\Qnt_{A, \Gamma}$ if and only if its associated graded is a universal element of $\PD_{A, \Gamma}$. 
\item The universal base $B$ of $\Qnt_{A,\Gamma}$ is a split-filtered commutative algebra, and there is a unipotent algebraic group $U$ consisting of automorphisms of $B$ which induce the identity on the associated graded $\gr B$. 
This group acts freely and transitively on the universal elements of $\Qnt_{A,\Gamma}$ over $B$ with a fixed associated graded. 
\end{enumerate}
\end{Theorem}
The proof will be completed in Section~\ref{ss:conicalsingulairities}. 
We expect even the case of $\Gamma$ trivial to have various applications, some of which we explain at the end of the introduction. 

We now return to the Lie theoretic setting discussed at the beginning of this introduction, and retain the notation introduced there.
We briefly recount the basic properties of the {\it Slodowy slice} $\Ss_\chi$.
Let $(e,h,f)$ be an $\sl_2$-triple and $\g^f$ the centraliser of $f$. Then $\Ss_\chi := \chi + \kappa(\g^f)$ is a transversal slice to coadjoint $G$-orbits, admitting a contracting $\C^\times$-action. Furthermore it carries a Poisson structure via Hamiltonian reduction, studied in \cite{GG} (see Section~\ref{ss:Poissonslices} for more detail). Slodowy showed \cite[Corollary~7.4.1]{Slo} that the adjoint quotient map $\g^* \to \g^*/\!/G \simeq \h^*/W$ restricts to a flat $\C^\times$-equivariant morphism $\Ss_\chi \to \h^*/W$ and it follows that the slice provides a Poisson deformation of the central fibre $\Ss_{\chi, \N} = \Ss_\chi \cap \N(\g^*)$ which we call the {\it nilpotent Slodowy slice}.

On the other hand, Premet introduced a filtered quantization of the Slodowy slice known as the finite $W$-algebra \cite{PrST}. This is a non-commutative filtered algebra $U(\g,e)$ which depends only on $\g$ and the orbit of $e$. Since their inception they have found numerous applications to the ordinary and modular representation theory of Lie algebras; see \cite{PrMF} for a nice overview of both settings. If $\lambda$ denotes a central character of $U(\g,e)$ then $U(\g,e)^\lambda$ will denote the quotient by the ideal generated by the kernel of $\lambda$.
\begin{Theorem}
\label{T:secondmain}
Let $A = \C[\Ss_{\chi, \N}]$. The following are equivalent:
\begin{enumerate}[label=(\arabic*)]
\item The functor $\Qnt_A$ is representable over $\C[\h^*/W]$ and the finite $W$-algebra $U(\g,e) \in \Qnt_A(\C[\h^*/W])$ is a universal element;
\item the orbit $\O$ is not listed in the Table 1.
\end{enumerate}
When these equivalent conditions hold  each filtered algebra quantizing $\C[\Ss_{\chi, \N}]$ is isomorphic  to $U(\g,e)^\lambda$ for some choice of central character $\lambda$.
\end{Theorem}
This result is proven in Theorem~\ref{T:Walgfilteredquant}. Our argument consists of assembling the necessary ingredients from throughout the literature to apply Theorem~\ref{T:firstmain} with $\Gamma$ trivial. In Section~\ref{ss:Poissonslices} we define the Poisson structure on $\Ss_\chi$ and recall the fact, well-known to experts, that $\Ss_{\chi, \N}$ is a conic symplectic singularity. In Section~\ref{ss:finiteWalgebras} we recall the basic properties of the finite $W$-algebra, whilst in Sections~\ref{ss:casimirsandcentres} and \ref{ss:universaldeformationoftheslice} we record the remaining facts needed to explain that the Slodowy slice and the finite $W$-algebra are Poisson deformations and quantizations of $\Ss_{\chi, \N}$.

Finally we focus on the subregular case considered by Brieskorn, Grothendieck and Slodowy \cite{Br, Slo}.
The non-commutative analogue of Brieskorn's theorem says that the subregular finite $W$-algebra attached to a simply laced Lie algebra of type $\Delta$ is a universal filtered quantization of the rational singularity of type $\Delta$. This is a special case of Theorem~\ref{T:secondmain}.
Some of the most interesting applications in this paper arise from our non-commutative analogue of Slodowy's theorem. Let $\g_0$ be a simple Lie algebra with non-simply-laced Dynkin diagram $\Delta_0$, and let ($\Delta, \Gamma_0)$ be determined by $\Delta_0$ by folding, as we described earlier.
Let $\chi=\kappa(e)$, where $e\in \g$ is a subregular element.
With some restrictions on $\Delta_0$ we prove the following analogue of Slodowy's theorem \cite[\textsection 8.8]{Slo} in the setting of universal Poisson deformations and their quantizations.
 \begin{Theorem}
 \label{T:thirdmain}
 Let $\g_0$ be of type {\sf B}$_n$, {\sf C}$_n$ or {\sf F}$_4$, where $n \ge 2$ and $n$ is even in type {\sf C}. Let $W_0$ be the Weyl group of $\g_0$, let $e_0 \in \N(\g_0)$ be a subregular nilpotent element and $\chi_0 = \kappa(e_0)$. Let $A = \C[\Ss_{\chi, \N}]$. Then: 
\begin{enumerate}[label=(\arabic*)]
\item The functor $\PD_{A,\Gamma_0}$ is representable over $\C[\h_0^*]^{W_0}$ and $\C[\Ss_{\chi_0}]$ is a universal element; 
\item The functor $\Qnt_{A,\Gamma_0}$ is representable over $\C[\h_0^*]^{W_0}$ and the 
finite $W$-algebra $U(\g_0, e_0)$ is a universal element.
\end{enumerate}
 \end{Theorem}
After laying the groundwork for the $\Gamma_0$-action in Section~\ref{ss:automorphismsandsubregular} the proof of this theorem is given in Section~\ref{ss:equivariantuniversal}. We expect that the restrictions on $\Delta_0$ are unnecessary and we conjecture that Theorem~\ref{T:thirdmain} holds for all non-simply-laced simple Lie algebras. Theorems~\ref{T:secondmain} and \ref{T:thirdmain} lead to interesting surjective homomorphisms between $W$-algebras, which are new in the literature.
\begin{Corollary}
\label{C:corollarytothethirdmain}
There exists a surjective homomorphism of subregular $W$-algebras
\begin{eqnarray}
U(\g, e) \twoheadrightarrow U(\g_0, e_0).
\end{eqnarray}
When $\g_0$ satisfies the hypotheses of Theorem~\ref{T:thirdmain}, the kernel is generated by elements $z - z\cdot \gamma$ where $\gamma\in \Gamma_0$ and $z\in Z(\g,e)$ lies in the centre.
\end{Corollary}

In the final section we apply Corollary~\ref{C:corollarytothethirdmain} to obtain new Yangian-type presentations of $W$-algebras. Since $W$-algebras are defined via quantum Hamiltonian reduction, there is no known presentation in general.
This makes them difficult to work with.
The situation is significantly improved in type {\sf A}: in this case there is an explicit isomorphism between a truncated shifted Yangian and the $W$-algebra \cite{BKshift}.
This leads to an explicit presentation of the subregular finite $W$-algebra in type {\sf B} as a quotient of a truncated shifted Yangian.
\begin{Theorem}
If $e_0 \in \g_0:=\so_{2n+1}$ lies in the subregular orbit then the finite $W$-algebra $U(\g_0, e_0)$ is generated by elements
\begin{eqnarray}
\begin{array}{c}
\{D_1^{(r)}, D_2^{(r)} \mid r>0\} \cup \{E^{(r)} \mid r> 2n-2 \} \cup \{F^{(r)} \mid r> 0\}
\end{array}
\end{eqnarray}
together with relations given in \cite[(2.4)-(2.9)]{BKshift}, along with the relations $D_1^{(r)} = 0 = Z^{(2s - 1)}$ where $r > 1$, $s=1,...,n$, and 
where $Z^{(r)}$ is given by the formula \eqref{e:centreYangian} below.
\end{Theorem}

\subsection{Further applications} We discuss three large families of conic symplectic singularities to which Theorem~\ref{T:firstmain} might be applied in the case $\Gamma = \{1\}$.
\begin{enumerate}
\item[(i)] Braverman--Finkelberg--Nakajima (BFN) have recently introduced a rigorous mathematical definition of the Coulomb branch associated to a class of 3-dimensional $\N = 4$ supersymmetric gauge theories \cite{BFN}. These are algebraic varieties attached to a pair $(G, V)$ where $G$ is a reductive group and $V$ is a $G$-module, which can be naturally quantized,  leading to a Poisson structure. The question of when these varieties have conical symplectic singularities has been a subject of much recent research [27, 35], and a positive answer has recently been given by Gwyn Bellamy \cite{Bel23}. These varieties also come equipped with natural deformations (see \textsection 2(ii) and 3(viii) of {\it loc. cit.}) and whenever the BFN deformation is universal in the sense of Namikawa, our theorem implies that the BFN quantization of the deformation is universal in the sense of Losev.

\item[(ii)] The universal quantization of a symplectic quotient singularity was determined by Losev \cite[Proposition~3.17]{Lo}. It can be constructed by taking the invariants in the spherical symplectic reflection algebra with respect to the Namikawa--Weyl group. Bellamy has previously obtained the Poisson analogue of this theorem \cite[Theorem~1.4]{Bel}, and combining our Theorem~\ref{T:firstmain} with {\it loc. cit.} one  obtains a (less detailed) proof of Losev's description of the universal quantization.

\item[(iii)] Other examples are expected to arise from symplectic reduction. To be more precise, suppose that we have a Hamiltonian action of a reductive group $G$ on a symplectic variety. Suppose furthermore that the Kirwan map on cohomology groups is an isomorphism (Cf. \cite{MN}). Then varying the moment map over the cocentre of $\Lie G$ should give the universal Poisson deformation of the symplectic reduction, and our Theorem~\ref{T:firstmain} suggests a straightforward proof that the quantum symplectic reduction is universal in the sense of Losev. We hope that examples of this phenomenon should be provided by quiver varieties and hypertoric varieties.
\end{enumerate}

\noindent {\bf Acknowledgements:} The authors would like to thank Gwyn Bellamy, Anne Moreau, Travis Schedler and Ben Webster for useful discussions on the subjects of this paper.
The authors thank the referee for careful review and suggestions.
 The research of the first three authors was partially supported by BIRD179758/17 Project ``Stratifications in algebraic groups, spherical varieties, Kac-Moody algebras and Kac-Moody groups''  and DOR1898721/18 ``Sheet e classi di Jordan in gruppi algebrici ed algebre di Lie'' funded by the University of Padova. 
The first author was also supported by FNS 200020{\_}175571, funded by the Swiss National Science Foundation.
The first three authors are members of the INDAM group GNSAGA.
The fourth author is supported  by the UKRI FLF MR/S032657/1, MR/S032657/2, MR/S032657/3 ``Geometric representation theory and $W$-algebras''.

\section{Universal Poisson deformations and filtered quantizations}

\subsection{Representability of functors and universal elements} \label{ss:cattheory}
Let $\mathbf{C}$ be a category. 
For an object $B \in \mathbf{C}$ we denote by $\Aut_{\mathbf{C}}(B)$ the set of automorphisms of $B$ in $\mathbf{C}$.
Let $\mathrm{F} \colon \mathbf{C} \to \mathbf{Sets}$ be a functor.
The functor $\mathrm{F}$ is said to be \emph{representable} if there exists a pair $(B_u, \eta)$ with $B_u \in \mathbf{C}$ and $\eta$  a natural isomorphism of functors $\eta \colon \Hom_{\mathbf{C}}(B_u, -) \to \mathrm{F}(-)$.
 Later in the paper we sometimes express this by saying that {\it $\mathrm{F}(-)$ is representable over $B_u$}.
In this case  $\mathrm{F}$ is said to be \emph{represented by $(B_u, \eta)$} and  $B_u \in \mathbf{C}$ is called a \emph{universal object (or base)} for $\mathrm{F}$, while $\eta$ is called a \emph{representation over $B_u$} of the functor $\mathrm{F}$.
A \emph{universal element} of $\mathrm{F}$ is a datum $(B_u, a)$ of  a universal object $B_u \in \mathbf{C}$ and an element $a \in \mathrm{F}(B_u)$ such that for all $(B, b)$ with $B \in \mathbf{C}$ and $b \in  \mathrm{F}(B)$ there exists a unique morphism $\phi \in \Hom_{\mathbf{C}}(B_u, B)$ satisfying $\mathrm{F}(\phi)(a) = b$,  \cite[III.1]{MacLane}.
For expository convenience, our terminology will slightly differ from the classical one: we will often refer to such an $a \in \mathrm{F}(B_u)$ as a universal element of $\mathrm{F}$, omitting the base when clear from the context.

By Yoneda's Lemma, the representations $\eta$ of $\mathrm{F}$ over a base $B_u$ correspond bijectively to the universal elements of $\mathrm{F}$ in $\mathrm{F}(B_u)$ via the map $\eta \mapsto \eta_{B_u}(\id_{B_u})$.
Once a universal element in $\mathrm{F}(B_u)$ is fixed, the collection of all universal elements of $\mathrm{F}$ in $\mathrm{F}(B_u)$ corresponds bijectively to $\Aut_{\mathbf{C}}(B_u)$;
in other words, the set of universal elements is an $\Aut_{\mathbf{C}}(B_u)$-torsor.

\subsection{Graded and filtered algebras}
\label{ss:grfltalg}

All vector spaces in this paper will be $\C$-vector spaces.
 Unadorned tensor products should be read as tensors over $\C$.
Every algebra in this paper is finitely generated, unless stated otherwise.

When we refer to a filtered vector space, we always mean a  filtration by the non-negative integers satisfying $V = \bigcup_{i \ge 0} V_i$, $V_0 = \C$ and $\dim(V_i) < \infty$ for all $i$. Similar hypotheses are assumed for all graded vector spaces. From a filtered vector space $V = \bigcup_{i\ge 0} V_i$ we can construct the associated graded space $\gr V = \bigoplus_{i \ge 0} V_i / V_{i-1}$, where we take $V_{-1} = 0$ by convention. A graded module over a graded algebra is free graded if it has a homogeneous basis.

We say that a filtered map of filtered vector spaces $\phi \colon V \to W$ is {\it strictly filtered} or {\it strict} if $\phi(V_i) = W_i \cap \phi(V)$. The importance of this definition is that $\gr$ is an exact functor from the category of filtered vector spaces with strict morphisms to the category of graded vector spaces \cite[Proposition~7.6.13]{MR} so that, for instance, a strict filtered embedding induces an embedding of associated graded vector spaces. In this paper every filtered morphism of vector spaces is assumed to be strict. 

When $V$ is a graded vector space we may regard it as filtered in the usual manner, and identify $V$ with $\gr V$ via the obvious splitting. Note that every graded map of graded vector spaces is a strictly filtered map. We shall often need to consider a map $\phi \colon V \to W$ from a graded space to a filtered space, and we call such a map strict if it is so when regarded as a map of filtered spaces. 

\begin{Lemma}
\label{L:flatequalsfree}
Let $A = \bigoplus_{i \ge 0} A_i$ be a finitely generated graded algebra with $A_0 = \C$ and let $M$ be a graded $A$-module. Then $M$ is flat if and only if $M$ is  a free graded module.
\end{Lemma}
\begin{proof} This follows directly from \cite[Lemma~2.2]{CM} because the grading on $A$ is  bounded from below.
\end{proof}

\begin{Lemma}
\label{L:grandtensors}
Suppose that $B$ is a commutative filtered algebra and that $C$ and $A$ are commutative filtered $B$-algebras such that the natural maps $B \to A$ and $B\to C$ are strictly filtered. Assume in addition that $\gr A$ is $\gr B$-flat.
There is a natural isomorphism
$$\gr A \otimes_{\gr B} \gr C\isoto \gr(A \otimes_B C)$$
\end{Lemma}
\begin{proof}
Since $\gr A$ is flat it is free graded by Lemma~\ref{L:flatequalsfree}. Hence, $A$ is a free object in the category of filtered $B$-modules \cite[Lemma 5.1, $3^\circ$]{NVO}. 
Consider the natural homomorphism of $\gr B$-modules $\varphi\colon \gr A \otimes_{\gr B} \gr C\onto \gr(A \otimes_B C)$ defined on homogeneous elements by $\varphi(\bar a\otimes \bar c)=\overline{a\otimes c}$, where we write $\bar v = v + V_{i-1} \in \gr V$ for the top graded component of $v \in V_i\setminus V_{i-1}$, with $V$ any filtered vector space.
 By \cite[Lemma 8.2]{NVO}, the map $\varphi$ is an isomorphism.
A direct verification shows that it is also an algebra homomorphism.
\end{proof}
Let $B, C, D$ be commutative rings and let $\A$ be a $B$-algebra. Suppose also that $C$ is a $B$-algebra and $D$ is a $C$-algebra. Then $D$ is naturally a $B$-algebra. We make the following notation
\begin{equation}\label{eq:cancel}
\varepsilon_{C} \colon \A \otimes_{B} C \otimes_{C} D \to \A \otimes_{B} D, \quad a \otimes c \otimes d \mapsto a \otimes cd.
\end{equation}

We denote by $\GrAlg$ the category of finitely generated non-negatively graded commutative $\C$-algebras $B = \bigoplus_{i\ge 0} B_i$ such that $B_0 = \C$.
Morphisms between objects of $\GrAlg$ are graded morphisms.
For $B \in \GrAlg$, we write $B_+$ for the unique graded maximal ideal, and $\C_+$ for the corresponding quotient $B/B_+ \simeq \C$.
Similarly, we denote by $\FAlg$ the category of finitely generated non-negatively filtered commutative $\C$-algebras $F$ such that $F_0 = \C$.
Morphisms between objects of $\FAlg$ are strictly filtered morphisms.

There is a functor $\filt \colon \GrAlg \to \FAlg$ which associates to a graded algebra $B$ (resp. a graded morphism) the filtered algebra $B$ with filtration induced by the grading (resp. the same morphism, considered as a filtered morphism).
We will consider the 
subcategory $\SFAlg$ of $\FAlg$ of split filtered algebras, i.e. the essential image of $\GrAlg$ through $\filt$.

Fix $n \in \Nbb$. A Poisson algebra is a commutative algebra $\A$ equipped with a Lie bracket $\{\cdot, \cdot\} \colon \A \times \A \to \A$ which is a biderivation. 
We say that a graded (resp. filtered) Poisson algebra $\A = \bigoplus_{i\ge 0}\A_i$ (resp. $\A = \bigcup_{i \ge 0} \A_i$) has Poisson bracket in degree $-n$ if $\{a,b\} \in \A_{i+j-n}$ for $a \in \A_i$, $b \in \A_j$. If $\A$ is a filtered Poisson algebra with bracket in degree $-n$ then $\gr \A$ is a graded Poisson algebra with bracket in degree $-n$. Similarly we say that a filtered (non-commutative) associative algebra $\A = \bigcup_{i\ge 0}\A_i$ has degree $-n$ if $[a,b] \in \A_{i+j-n}$ whenever $a \in \A_i$ and $b \in \A_j$. Such filtered algebras have the property that $\gr \A$ is a graded Poisson algebra in degree $-n$ via the formula
\begin{eqnarray}\label{eq:Poisson_bracket}
\{a + \A_{i-1}, b + \A_{j-1}\} := [a,b] + \A_{i + j - n - 1}
\end{eqnarray}
whenever $a \in \A_i$, $b \in \A_j$. Similarly, filtered homomorphisms between filtered algebras of degree $-n$ induce graded homomorphisms between Poisson algebras of degree $-n$. These observations can be upgraded to a well-known categorical statement.

\begin{Lemma}
\label{L:grisafunctor}
Taking the associated graded defines an exact functor from the category of strictly filtered algebras of degree $-n$ 
to the category of graded Poisson algebras of degree $-n$. $\hfill \qed$
\end{Lemma}

For $B \in \GrAlg$, when we say that $\A$ is a Poisson $B$-algebra we assume that the structure map $B \to \A$ is a homomorphism whose image is Poisson central.
Recall that, for $B \in \FAlg$ and $\A$ a filtered $B$-algebra, we always assume that the structure map $B \to \A$ is strictly filtered.

\subsection{The Poisson deformation functor} \label{ss:poissondefsection}
For the rest of the section we keep fixed $n\in \Nbb$ and a positively graded, finitely generated, commutative, integral Poisson algebra $A = \bigoplus_{i \ge 0} A_i$ with Poisson bracket in degree $-n$.

A {\it filtered Poisson deformation of $A$} is a pair $(\A, \iota)$ consisting of a filtered Poisson algebra with bracket in degree $-n$, and an isomorphism $\iota : \gr \A \to A$ of Poisson algebras.

By the Rees algebra construction, a filtered Poisson deformation is a special case of a more general notion of Poisson deformation.

\begin{Definition} \label{def_gradedPD}
Let $B \in \GrAlg$. A \emph{graded Poisson deformation of $A$ over $B$} is a pair $(\A, \iota)$ where:
\begin{enumerate}
\item[(i)] $\A$ is a graded Poisson $B$-algebra, flat as a $B$-module;
\item[(ii)] $\iota \colon \A \otimes_B \C_+ \to A$ is an isomorphism of graded Poisson algebras.
\end{enumerate}
\end{Definition}

Two graded Poisson deformations $(\A_1, \iota_1)$ and $(\A_2, \iota_2)$ of $A$ over $B$ are said to be \emph{isomorphic} if there exists a graded Poisson algebra isomorphism $\phi \colon \A_1 \to \A_2$ such that the following diagrams commute
\begin{eqnarray}
\label{e:commsquare}
\begin{tikzcd}
 &  B\arrow[dr] \arrow[dl] & \\
\A_1 \arrow{rr}{\phi} & & \A_2
\end{tikzcd}
\end{eqnarray}
\begin{eqnarray}
\label{e:iotadiagram}
\begin{tikzcd}
 &  A  \\
\A_1 \otimes_{B} \C_+ \arrow{rr}{\phi \otimes \id} \arrow{ur}{\iota_1} &&
\A_2 \otimes_{B} \C_+ \arrow{ul}[swap]{\iota_2}
\end{tikzcd}
\end{eqnarray}

\begin{Remark} \label{R:ses}
Let  $B$ and $(\A,\iota)$ be as in Definition \ref{def_gradedPD} and consider the (split) short exact sequence of $B$-modules $0 \to B_+ \to B \to \C_+ \to 0$.
Since $\A$ is flat over $B$, we obtain another short exact sequence of graded $B$-modules after tensoring with $\A$ over $B$ and using $\iota$: 
\begin{equation} \label{eq_ses}
0 \rightarrow B_+ \A \rightarrow \A \xrightarrow{\pi} A \rightarrow 0.
\end{equation}
Since $B$ is Poisson central in $\A$, $B_+ \A$ is a Poisson ideal in $\A$ and $\pi$ in \eqref{eq_ses} respects the Poisson brackets.
We choose a graded vector subspace $A' \subset \A$ isomorphic to $A$ as a graded vector space and such that $A' \oplus B_+ \A = \A$.
By \cite[Lemma 2.2]{CM}, we conclude that we have an isomorphism of graded $B$-modules $\A \simeq A' \otimes B$ (where $B$ acts on the right tensor factor) and the $n$-th graded component of \eqref{eq_ses} reads:
\begin{equation}
 \label{eq_ses_gr}
0 \to \bigoplus_{i > 0 }^n A'_{n-i} \otimes B_i \to \bigoplus_{i = 0}^n A'_{n-i} \otimes B_i \to A'_n \to 0.
\end{equation}
\end{Remark}

\begin{Definition} \label{def_GPDfunctor}
The {\em functor of graded Poisson deformations of $A$} is defined as follows:
$$\PD_A \colon \GrAlg \to \Sets$$
\begin{enumerate}[label=(\roman*)]
\item for $B \in \GrAlg$, the set $\PD_A(B)$ is the set of isoclasses of graded Poisson deformations $(\A, \iota)$ over $B$;
\item for $\beta \in \Hom_{\GrAlg}(B_1, B_2)$  we set $\PD_A(\beta)$ to be the map associating to the isoclass of
  $(\A_1, \iota_1) \in \PD_A(B_1)$ the isoclass of $(\A_1 \otimes_{B_1} B_2, \iota_1 \circ \varepsilon_{B_2}) \in \PD_A(B_2)$, where $ \varepsilon_{B_2}$ is defined in \eqref{eq:cancel}.
\end{enumerate}
\end{Definition}

When clear from the context, we will denote isoclasses by their representatives.

If  $\PD_A$ is representable, we use the following notation: we denote by $B_u \in \GrAlg$ a fixed universal base, we fix one of its representations over $B_u$ and we set $(\A_u, \iota_u)$ to be the element in $\PD_A(B_u)$ corresponding to $\id_{B_u}$ and call it a \emph{universal graded Poisson deformation of $A$}. 

For the reader's convenience, we explicitly state the universal property of $(\A_u, \iota_u)$, as in \cite[Proposition 2.12]{Lo}: for $B \in \GrAlg$ and $(\A,\iota) \in \PD_A(B)$, there exists a unique $\beta \in \Hom_{\GrAlg}(B_u, B)$ such that $\PD_A(\beta)(\A_u, \iota_u) = (\A, \iota)$ in $\PD_A(B)$, i.e. there exists a $B$-linear graded Poisson isomorphism $\phi \colon \A_u \otimes_{B_u} B \to \A$ such that the following diagram (see \eqref{e:iotadiagram}) commutes:
\begin{center}
\begin{tikzcd}                                                                            & A &                                                                       \\
\A_u \otimes_{B_u} B \otimes_B \C_+ \arrow[ru, "\iota_u \circ \varepsilon_B"] \arrow[rr, "\phi \otimes \id"]&   & \A \otimes_{B} \C_+ \arrow[lu, "\iota"'] 
\end{tikzcd}
\end{center}
where $\varepsilon_B$ is defined in \eqref{eq:cancel}.

\begin{Remark}\label{rk_iso_outside}
Note that the notation $\varepsilon_C$ in \eqref{eq:cancel} actually suppresses the choice of homomorphism $\varphi : B \to C$ making $C$ a $B$-algebra. In this Remark we write $\varepsilon_C := \varepsilon_{C, \varphi}$. Let $B \in \GrAlg$ and $\beta \in \Aut_{\GrAlg}(B)$ and consider $\PD_A(\beta) (\A, \iota) = (\A \otimes_B B, \iota \circ \varepsilon_{B, \beta}) \in \PD_A(B)$.
Then $\A$ and $\A \otimes_B B$ are isomorphic as graded Poisson algebras and both deform $A$ but $(\A, \iota)$ and $(\A \otimes_B B, \iota \circ \varepsilon_{B, \beta})$ are generally distinct elements of $\PD_A(B)$, since the isomorphism need not be $B$-linear, i.e., the diagram in \eqref{e:commsquare} may not commute.

Suppose $\PD_A$ is represented by $B_u \in \GrAlg$, then the choice of a universal graded Poisson deformation is not unique: the graded automorphisms of $B_u$ correspond bijectively to elements of $\PD_A(B_u)$, each of which can serve as a choice of universal element. See \S \ref{ss:cattheory} for further details.
%
\end{Remark}

\begin{Example}
\label{E:nullconedeformation}
Let $G$ be a simply-connected complex semisimple group, $\g$ its Lie algebra with Cartan subalgebra $\h$ and Weyl group $W$. The coordinate ring $\C[\g^*]$ is naturally a graded Poisson algebra, and the Poisson centre coincides with $\C[\g^*]^G$ which is isomorphic to $\C[\h^*]^W = \C[\h^*/W]$ by Chevalley's restriction theorem. One may choose the isomorphism to ensure that it is an isomorphism of graded algebras. 
Set $A = \C[\N(\g^*)]$ to be the coordinate ring of the nullcone.
This is graded via the contracting $\C^\times$-action on $\g^*$ and a famous theorem of Kostant \cite[Proposition~7.13]{JaNO} says that the defining ideal of $\N(\g^*)$ in $\C[\g^*]$ is generated by $\C[\g^*]^G_+$.
Hence $A$ is a positively graded Poisson algebra and there is an isomorphism $\iota\colon \C[\g^*] \otimes_{\C[\h^*/W]} \C_+ \to A$.
By another theorem of Kostant $\C[\g^*]$ is a free $\C[\g^*]^G$-module \cite[Theorem~0.2]{Ko1} and so $(\C[\g^*], \iota) \in \PD_{A}(\C[\h^*/W])$ is a Poisson deformation of the coordinate ring of the nullcone.
\end{Example}

\subsection{The filtered quantization functor}
\label{ss:quantizationdsection}
We continue to fix $n\in \Nbb$ and $A$, and we remind the reader that all filtered maps in this paper are assumed to be  strict.
Our goal here is to define a functor similar to $\PD_A$ classifying filtered quantizations of $A$.
Recall that for a filtered algebra of degree $-n$, the associated graded algebra carries a Poisson structure via \eqref{eq:Poisson_bracket}. 

\begin{Definition}\label{D:quant}
Let $B \in \SFAlg$. A {\it filtered quantization (of a Poisson deformation) of $A$ over $B$} is a pair $(\Q, \iota)$ where:
\begin{enumerate}
\item[(i)] $\Q$ is a filtered $B$-algebra of degree $-n$, flat as a $B$-module;
\item[(ii)] $(\gr \Q, \iota) \in \PD_A(\gr B)$, where the structure map  $\gr B \to \gr \Q$ is induced by the structure map $B \to \Q$.
\end{enumerate}
Once again, we call $B$ the {\it base} of the quantization.
\end{Definition}


Two quantizations $(\Q_1, \iota_1)$ and $(\Q_2, \iota_2)$ of $A$ over $B$ are said to be \emph{isomorphic}  if there exists a filtered isomorphism $\phi \colon \Q_1 \to \Q_2$ such that:
\begin{enumerate}
\item[(i)] the following diagram commutes:
\begin{eqnarray}
\label{e:qommsquare}
\begin{tikzcd}
& B \arrow[dr] \arrow[dl] &  \\
\Q_1 \arrow{rr}{\phi} & & \Q_2
\end{tikzcd}
\end{eqnarray}
\item[(ii)] $\gr \phi \colon \gr \Q_1 \to \gr \Q_2$ gives an isomorphism between $(\gr \Q_1, \iota_1)$ and $(\gr \Q_2, \iota_2)$ as graded Poisson deformations of $A$ over $B$.
\end{enumerate}

\begin{Definition} \label{D:qntfunct}
The {\em functor of filtered quantizations of $A$} is defined as follows:
$$\Qnt_A \colon \SFAlg \to \Sets$$
\begin{enumerate}[label=(\roman*)]
\item for $B \in \SFAlg$, the set $\Qnt_A(B)$ is the set of isoclasses of filtered quantizations $(\Q, \iota)$ over $B$;
\item for $\beta \in \Hom_{\SFAlg}(B_1, B_2)$, the morphism $\Qnt_A(\beta)$ maps the isoclass of $(\Q_1, \iota_1) \in \Qnt_A(B_1)$ to the isoclass of $(\Q_1 \otimes_{B_1} B_2, \iota_1 \circ \varepsilon_{\gr B_2}) \in \Qnt_A(B_2)$, where $\varepsilon_{\gr B_2}$ is defined in \eqref{eq:cancel}.
\end{enumerate}
\end{Definition}

When $\Qnt_A$ is representable, we denote by $B_u \in \SFAlg$ a fixed universal base, we fix one of its representations over $B_u$  and we set  $(\Q_u, \iota_u) \in \Qnt_A(B_u)$ to be the universal element corresponding to $\id_{B_u}$ and call it a \emph{universal filtered quantization of $A$}.

Then $(\Q_u, \iota_u) \in \Q_A(B_u)$ satisfies the universal property in \cite[Proposition 3.5]{Lo}, which we restate in the functorial language:
for all $B \in \SFAlg$ and  $(\Q, \iota) \in \Q_A(B)$, there exists a unique morphism $\beta \in \Hom_{\SFAlg}(B_u, B)$ such that
$(\Q, \iota) = \Qnt_A(\beta)(\Q_u, \iota_u) = (\Q_u \otimes_{B_u} B, \iota \circ \varepsilon_{\gr B})$ as elements in $\Qnt_{A}(B)$, where $\varepsilon_{\gr B}$ is defined in \eqref{eq:cancel}.

A quantum analogue of Remark \ref{rk_iso_outside} applies.

\begin{Example}
\label{E:nullconequantization}
Consider the enveloping algebra $U(\g)$ which is filtered of degree $-1$ with $\gr U(\g) \simeq\C[\g^*]$,  and as usual denote the centre of $U(\g)$ by $Z(\g)$. Thanks to the Harish--Chandra restriction theorem we know that there is a choice of grading on $Z(\g)$ such that we have graded isomorphisms $Z(\g) \simeq \C[\h^*]^{W_\bullet} = \C[\h^*/W_\bullet]\simeq\C[\h^*]^W$ (the filtered maps are clearly strictly filtered), see Section \ref{ss:casimirsandcentres} for a more detailed account.
 Furthermore $\gr Z(\g)\simeq Z(\g)$ identifies with $\C[\g^*]^G \subseteq \C[\g^*]$ as a subalgebra of $\gr U(\g)$. 

In virtue of Example~\ref{E:nullconedeformation}, from which we retain notation, we have $(U(\g),  \iota) \in \Qnt_A(\C[\h^*/W_\bullet])$, with $A = \C[\N(\g^*)]$. 
The construction depends on a choice of grading on $Z(\g)$, which is the same as fixing a choice of strictly filtered isomorphism $\C[\h^*]^W \to Z(\g)$ and these various choices of isomorphism lead to different isomorphism classes of quantization of $A$.
\end{Example}

\subsection{Interplay between deformations and quantizations} \label{ss:interplay}
There is a natural relationship between the two functors we have introduced above: this is given by the natural transformation $\graded \colon \Qnt_A \to (\PD_A \circ \gr)$. For $B \in \SFAlg$, the map $\graded_B \colon \Qnt_A(B) \to \PD_A (\gr B)$ is defined by $(\Q, \iota) \mapsto (\gr \Q, \iota)$.
This relation is extremely nice under the following assumptions, which hold for conic symplectic singularities (introduced by Beauville \cite{Be}): these are the main family of Poisson varieties to which we intend to apply the results of this paper.

\begin{Definition} \label{D:oc}
We say that $A$ admits an \emph{optimal quantization theory} (OQT) if the following conditions are satisfied:
\begin{enumerate}[label=(\roman*)]
\item The functor $\Qnt_A$ is representable over $B_u \in \SFAlg$;
\item The functor $\PD_A$ is representable over $\gr B_u \in \GrAlg$;
\item There exist a representation $\eta$ of $\Qnt_A$ over $B_u$, resp. $\zeta$ of $\PD_A$ over $\gr B_u$, such that the following diagram of natural transformations commutes:
\begin{center}
\begin{tikzcd}
\Hom_{\SFAlg}(B_u, -)   \arrow[d,,"\graded"']  \arrow[r, "\eta"] & \Qnt_{A}(-) \arrow[d,,"\graded"] \\
  \Hom_{\GrAlg}(\gr B_u, \gr -) \arrow[r, "\zeta"] &  \PD_A \gr(-)
\end{tikzcd}
\end{center}
where, for $B \in \SFAlg$, we set $\graded_B \colon \Hom_{\SFAlg}(B_u, B) \to \Hom_{\GrAlg}(\gr B_u, \gr B)$ to be the map $\beta \mapsto \gr \beta$.
\end{enumerate}
\end{Definition} 
\begin{Remark}\label{R:oc}
An OQT is equivalent to the existence of universal bases $B_u$ of $\Qnt_A$, resp. $C_u$ of $\PD_A$, and universal elements $(\Q_u, \iota_{\Q}) \in \Qnt_A(B_u)$, resp. $(\A_u, \iota_{\A}) \in \PD_A(C_u)$, and a graded isomorphism $\phi \in \Hom_{\GrAlg}(C_u, \gr B_u)$ (which is necessarily unique) such that $\PD_A(\phi)(\A_u, \iota_{\A}) = (\gr \Q_u, \iota_{\Q})$.
\end{Remark}

\begin{Lemma}\label{L:interplay}
Assume $A$ admits an OQT. Let $B \in \SFAlg$ and $(\Q, \iota) \in \Qnt_A(B)$. Suppose that $(\gr \Q, \iota) \in \PD_A(\gr B)$ is a universal element of $\PD_A$. Then $(\Q, \iota)$ is a universal element of $\Qnt_A$.
\end{Lemma}
\begin{proof}
Consider the universal base $B_u \in \SFAlg$ from OQT axiom (i) and fix a universal element $(\Q_u, \iota_u) \in \Qnt_A(B_u)$. 
There exists a unique $\varphi \in \Hom_{\SFAlg}(B_u, B)$ such that $\Qnt_A(\varphi)(\Q_u, \iota_u) = (\Q, \iota)$.
Now OQT axiom (iii) implies that $(\gr \Q, \iota) = \graded(\Qnt_A(\varphi))(\Q_u, \iota_u) = \PD_A(\gr \varphi)(\gr \Q_u, \iota_u)$.
By OQT (ii), there exists a unique $\beta \in \Hom_{\GrAlg}(\gr B_u, \gr B)$ such that $\PD_A(\beta)(\gr \Q_u, \iota_\Q) = (\gr \Q, \iota)$ and $\beta$ is an isomorphism, by the assumption of universality of $(\gr \Q, \iota)$.
Uniqueness yields $\beta = \gr \varphi$.
We conclude that $\varphi$ is an isomorphism, hence $(\Q, \iota)$ is a universal element of $\Qnt_A$.
\end{proof}

\subsection{Automorphisms of $A$ and of its deformations} \label{ss_autdef}

Denote by $\PAut_{\GrAlg}(A)$  the group of graded Poisson automorphisms of $A$. 
For $B \in \GrAlg$ and $(\A, \iota) \in \PD_A(B)$, we denote by $\Aut_{B-\GrAlg}(\A)$ the group of $B$-linear graded automorphisms of $\A$.
Moreover, we write $\PAut_{B-\GrAlg}(\A)$ for the subgroup of automorphisms in $\Aut_{B-\GrAlg}(\A)$ respecting the Poisson brackets.

We retain  the surjection $\pi \colon \A \to \A/B_+ \A$ from Remark \ref{R:ses} and we identify  the algebras $\A/B_+ \A \simeq A \otimes_B \C_+ \simeq A$ through $\iota$.
We define a map
\begin{equation}\label{eq:p}
p \colon \PAut_{B-\GrAlg}(\A) \to \PAut_{\GrAlg}(A).
\end{equation}
by setting $p (\gamma) \in \PAut_{\GrAlg}(A)$ to be the map defined by $p (\gamma)(a+ B_+\A):=\gamma(a)+B_+\A$ for $a\in\A$.
Since the elements of $\PAut_{B-\GrAlg}(\A)$ preserve the Poisson ideal $B_+ \A$, for $\gamma \in \PAut_{B-\GrAlg}(\A)$, the map is well-defined.

\begin{Lemma}\label{L:propertyp}
Let $B \in \GrAlg$ and $(\A, \iota) \in \PD_A(B)$. Then the following hold:
\begin{enumerate}[label=(\arabic*)]
\item $\PAut_{\GrAlg}(A)$ and $\PAut_{B-\GrAlg}(\A)$ are linear algebraic groups.
\item The map $p$ defined in \eqref{eq:p} is a morphism of linear algebraic groups.
\item The kernel of $p$ is  unipotent and $\ker p = \PAut_{B - \GrAlg}(\A, \iota)$,  the automorphism group of the graded Poisson deformation $(\A,\iota)$ over $B$.
\end{enumerate}
\end{Lemma}
\begin{proof}
Since $A$ is a finitely generated algebra, 
$\Aut_{\GrAlg}(A)$ can be identified with a closed subgroup of $\GL(V)$ for some finite dimensional vector space $V$, i.e.  $\Aut_{\GrAlg}(A)$
is a (possibly disconnected) linear algebraic group; see also \cite[\S 3.7]{Lo}.
By a similar argument, $\Aut_{B-\GrAlg}(\A)$ is a linear algebraic group.
Since automorphisms respecting the Poisson brackets define a Zariski-closed subset, we have proven (1).

The map $p$ is clearly a group homomorphism.
We prove that it is a morphism of algebraic varieties: there is  $N\geq 0$ such that  $\Aut_{B-\GrAlg}(\A)$ and $\Aut_{B-\GrAlg}(A)$ can be identified as closed subgroups of
$\GL(\bigoplus_{j=0}^N\A_j)$ and $\GL(\bigoplus_{j=0}^NA_j)$, respectively.
Now, as vector spaces $\bigoplus_{j=0}^N\A_j \simeq  (\bigoplus_{j=0}^N\A_j \cap B_+\A)\oplus \bigoplus_{j=0}^NA_j$ and $p$ maps the automorphism $\gamma$ to its restriction to $\bigoplus_{j=0}^NA_j$. This completes the proof of (2).

Finally, we describe $\ker p$.  Observe that $\gamma\in \ker p$ if and only if $\iota\circ (\gamma\otimes \id)=\iota$, that is,  $\gamma$ is a $B$-linear Poisson graded automorphism of $\A$ inducing the identity on the central fibre $A$.
In particular, and in view of Remark \ref{R:ses}, the map $\gamma$ is an automorphism of the graded $B$-module $A \otimes B$ and it satisfies $\pi \circ \gamma = \pi$. 
Combining $\ker \pi = B_+ \A$ with an induction argument on the graded components of $\gamma$ proves that $\gamma$ is unipotent, hence $\ker p$ is unipotent. This concludes the proof of (3).
\end{proof}

The algebraic group $\PAut_{\GrAlg}(A)$ admits a Levi decomposition by Mostow's Theorem, see \cite[VIII, Theorem 4.3]{Hoch}: we denote by $\G$ the (possibly disconnected) reductive part of $\PAut_{\GrAlg}(A)$.
Throughout this section, we fix a reductive subgroup $\Gamma \leq \G$.

Let $B \in \GrAlg$. For any $(\A, \iota) \in \PD_A(B)$, we define the  subgroup of  $\PAut_{B-\GrAlg}(\A)$:
\begin{equation}\label{eq:gammatilde}
\widetilde{\Gamma}(\A, \iota) := \{\gamma \in \PAut_{B-\GrAlg}(\A) \mid  \iota \circ (\gamma \otimes \id) = g \circ \iota \mbox{ for some } g \in \Gamma\}.
\end{equation}
The group \eqref{eq:gammatilde} can be described quite transparently with fewer symbols: these are the $B$-linear graded Poisson automorphisms of $\A$ which restrict to elements of $\Gamma$ on the central fibre $A$; to phrase it another way, we consider automorphisms $\gamma$ of $\A$ such that $\pi \circ \gamma = g \circ \pi$ for some $g \in \Gamma$.

For $B \in \GrAlg$, $(\A, \iota) \in \PD_A(B)$ and $g \in \Gamma$, we define the $g$-twisted deformation $^g(\A, \iota) = (\A, g \circ \iota )$: this yields a (left) action of $\Gamma$ on the set $\PD_A(B)$.
We remark that the existence of $\gamma\in\widetilde{\Gamma}(\A, \iota)$ restricting to $g\in\Gamma$ on $A$ is equivalent to the condition  $^g(\A,\iota)=(\A,\iota)$ in $\PD_A(B)$.

\begin{Definition}\label{D:gammastabledef}
Let $B \in \GrAlg$.
A \emph{$\Gamma$-stable Poisson graded deformation of $A$ over $B$} is an element $(\A, \iota) \in \PD_A(B)$ such that 
$^g{(\A, \iota)} = (\A, \iota)$ as elements of $\PD_A(B)$, for all $g \in \Gamma$.

For $B \in \GrAlg$, let $\PD_A(B)^\Gamma$ denote the subset of isoclasses of $\Gamma$-stable graded deformation of $A$ over $B$.
Then for $B' \in \GrAlg$ and $\beta \in \Hom_{\GrAlg}(B, B')$, we have $\PD_A(\beta) (\PD_A(B)^\Gamma) \subset \PD_A(B')^\Gamma$.
Hence we can define \emph{the functor of $\Gamma$-stable Poisson graded deformations of $A$}: $$\PD_A^\Gamma \colon \GrAlg \to \Sets$$ associating to  $B \in \GrAlg$ the set $\PD^\Gamma_A(B) := \PD_A(B)^\Gamma$ and operating on morphisms as per Definition \ref{def_GPDfunctor}(ii).
\end{Definition}

\begin{Lemma} \label{L:semidirect}
Let $B \in \GrAlg$, $(\A,\iota) \in \PD^\Gamma_A(B)$ and let $\widetilde{\Gamma} := \widetilde{\Gamma}(\A, \iota)$. 
Then $\widetilde \Gamma$ is a linear algebraic group with unipotent radical $U =\ker p$ and Levi decomposition $\widetilde{\Gamma} \simeq \Gamma \ltimes U$.
\end{Lemma}

\begin{proof}
%
%
%
Let $p$ be as in \eqref{eq:p}, then the group $\widetilde \Gamma$ defined in \eqref{eq:qgammatilde} is precisely $p^{-1}(\Gamma)$.
This is nonempty, because $(\A,\iota) \in \PD^\Gamma_A(B)$ and it is a closed (hence algebraic) subgroup of $\PAut_{B- \GrAlg}(\A)$.
We set $U := \ker p$ and observe that it is contained in $\tilde \Gamma$.
The last assertion follows from Lemma \ref{L:propertyp} (3), the short exact sequence $1 \to U \to \widetilde \Gamma \xrightarrow{p} \Gamma \to 1$ and Mostow's Theorem \cite[VIII, Theorem 4.3]{Hoch}.
\end{proof}


\begin{Remark} \label{R:univisequiv}
Suppose that $\PD_A$ is representable and retain notation from \S \ref{ss:poissondefsection}.
Let $g \in \Gamma$, and consider $(\A_u, g \circ \iota_u) \in \PD_A(B_u)$.
By representability, there exists a unique morphism $\alpha_g \in \Hom_{\GrAlg}(B_u,B_u)$ such that $\PD_A(\alpha_g)(\A_u,  \iota_u) = (\A_u, g \circ \iota_u)$. Thus, for each $g\in \Gamma$, there is an isomorphism of deformations $\widetilde{\alpha}_g\colon \A_u\otimes_{B_u}B_u\to \A_u$ inducing the equality $\PD_A(\alpha_g)(\A_u,  \iota_u) = (\A_u, g \circ \iota_u)$.
We claim that the collection of morphisms $(\alpha_g)_{g \in \Gamma}$ yields a right action of $\Gamma$ on $B_u$. Indeed, for $h, g \in \Gamma$, 
we  have:
\begin{align*}
\PD_A(\alpha_h) \PD_A(\alpha_g) (\A_u, \iota_u) & =  \PD_A(\alpha_h)(\A_u,  g \circ \iota_u)  =(\A_u \otimes_{B_u} B_u,  g \circ \iota_u \circ \varepsilon_{B_u})\\
& = (\A_u \otimes_{B_u} B_u,  g \circ h\circ \iota_u \circ (\widetilde{\alpha}_h\otimes \id))= (\A_u, g \circ h \circ \iota_u)  =\PD(\alpha_{gh})(\A_u, \iota_u),
\end{align*}
where  $\varepsilon_{B_u}$ is as in \eqref{eq:cancel} and  all equalities should be read as equalities of isoclasses.
Uniqueness of $\alpha_g, \alpha_h$ and $\alpha_{gh}$ yields the proof of the claim.
\end{Remark}

\begin{Definition}\label{D:equivstr_def}
Let $B \in \GrAlg$ and $(\A, \iota) \in \PD_A(B)$.
Set $\widetilde{\Gamma}:= \widetilde{\Gamma}(\A, \iota)$, as in \S \ref{ss_autdef}.
A \emph{$\Gamma$-structure on $(\A, \iota)$} is a group morphism $s \colon \Gamma \to \widetilde{\Gamma}$ such that $p \circ s = \id_\Gamma$.
\end{Definition}
Namely, the morphism $s \colon \Gamma \to \widetilde{\Gamma}$ gives a family of commutative diagrams, for $g \in \Gamma$:
\begin{center}
\begin{tikzcd}
A \arrow[r, "g"] & A  \\
\A   \otimes_{B} \C_+  \arrow[r, "s(g) \otimes \id"] \arrow{u}{\iota} &
\A   \otimes_{B} \C_+  \arrow[u,"\iota"']
\end{tikzcd}
\end{center}

\begin{Definition} \label{def_gammagpd}
Let $B \in \GrAlg$.
A \emph{graded Poisson $\Gamma$-deformation of $A$ over $B$} is a triple $(\A, \iota, s)$
such that:
\begin{enumerate}[label=(\roman*)]
\item $(\A, \iota) \in \PD_A(B)$;
\item $s$ is a $\Gamma$-structure on $(\A, \iota)$.
\end{enumerate}
\end{Definition}

We say that two  graded Poisson $\Gamma$-deformations $(\A_1, \iota_1, s_1)$ and $(\A_2, \iota_2, s_2)$ of $A$ over $B$ are \emph{isomorphic} if  there exists a  graded Poisson algebra isomorphism $\phi \colon \A_1 \to \A_2$ such that \eqref{e:commsquare} and \eqref{e:iotadiagram} commute, and $\phi \circ s_1(g) = s_2(g) \circ \phi$ for all $g \in \Gamma$.

\begin{Definition} \label{def_gammafunctor}
The {\em functor of graded Poisson $\Gamma$-deformations of $A$} is defined as follows:
$$\PD_{A, \Gamma} \colon  \GrAlg \to \Sets$$
\begin{enumerate}[label=(\roman*)]
\item for $B \in \GrAlg$, the set $\PD_{A,\Gamma}(B)$ consists of isoclasses of graded Poisson $\Gamma$-deformations $(\A, \iota, s)$ over $B$;
\item for $\beta \in \Hom_{\GrAlg}(B_1, B_2)$, the morphism $\PD_{A,\Gamma}(\beta)$ maps $(\A_1, \iota_1, s_1) \in \PD_{A,\Gamma}(B_1)$ to $(\A_2, \iota_2, s_2) \in \PD_{A,\Gamma}(B_2)$ where $(\A_2,\iota_2) = \PD_A(\beta)(\A_1, \iota_1)$ and, for $g \in \Gamma$ we set $s_2(g) = s_1(g) \otimes \id$.
\end{enumerate}
\end{Definition}

We define the forgetful natural transformation
$\Forg \colon \PD_{A,\Gamma} \to \PD_A$ as follows: for $B \in \GrAlg$, the component $\Forg_B \colon \PD_{A,\Gamma}(B) \to \PD_A(B)$ maps $(\A, \iota, s) \mapsto (\A, \iota)$.
The following is one of the most important general results of this paper.
\begin{Theorem} \label{T:dforgetful} 
The natural transformation $\Forg \colon \PD_{A,\Gamma} \to \PD_A$ factors through $\PD_A^\Gamma$ and induces a natural isomorphism between the functors $\PD_{A, \Gamma}$ and $\PD_A^\Gamma$.
\end{Theorem}
\begin{proof}
Let $B \in \GrAlg$.
The image of $\Forg_B$ is contained in $\PD_A^\Gamma(B)$. 
Indeed, by definition of $\Gamma$-structure, for each $g \in \Gamma$, the graded Poisson automorphism $s(g)$ is $B$-linear and satisfies $\iota \circ (s(g) \otimes \id) = g \circ \iota$, i.e. \eqref{e:commsquare} and \eqref{e:iotadiagram} commute.

We prove that $\Forg_B \colon  \PD_{A, \Gamma}(B) \to \PD_A^\Gamma(B)$ is surjective.
Let $(\A,\iota)$ be an object in $\PD_A^\Gamma(B)$ and consider $p$ from \eqref{eq:p}.
 By Lemma \ref{L:semidirect}, from which we retain notation, the map $p$ restricts to a surjective algebraic group morphism  $p\colon \widetilde{\Gamma}(\A,\iota)\simeq \Gamma\ltimes U\to \Gamma$.
 Any section $s\colon\Gamma\to \widetilde{\Gamma}(\A,\iota)$ thus endows $(\A,\iota)$ with a $\Gamma$-structure.

Finally, we prove that $\Forg_B \colon  \PD_{A, \Gamma}(B) \to \PD_A^\Gamma(B)$ is injective. 
For $i =1,2$, let $(\A_i, \iota_i, s_i) \in \PD_{A, \Gamma}(B)$ such that $(\A_1, \iota_1)$ and $(\A_2, \iota_2)$ are isomorphic graded Poisson deformations over $B$.
Without loss of generality, we can assume that $(\A_i, \iota_i) = (\A, \iota)$ for $i = 1,2$. 
Now,  by Lemma \ref{L:semidirect} and  Lemma \ref{L:propertyp} (3),  all sections of $p$, that is, all splittings of the short exact sequence $1 \to U \to \widetilde{\Gamma} \xrightarrow{p} \Gamma \to 1$ are conjugate by an element of $U$. 
Hence, there is $\phi \in \PAut_{B-\GrAlg}(\A)$ conjugating the $\Gamma$-structure $s_1$ to $s_2$, i.e. $(\A, \iota, s_1)$ and  $(\A, \iota, s_2)$ are isomorphic as graded $\Gamma$-deformations.
\end{proof}

As a consequence  we can identify the functors $\PD_{A,\Gamma}$ and $\PD_A^\Gamma$; in particular, for $B \in \GrAlg$, we may omit the $\Gamma$-structure when writing an element of $\PD_{A,\Gamma}(B)$.
%

We recall the right $\Gamma$-action on $B_u$ defined in Remark \ref{R:univisequiv} and we set $(B_u)_{\Gamma}$ to be the quotient of $B_u$ modulo the homogeneous ideal generated by $\{b - b \cdot \gamma \mid b \in B_u, \gamma \in \Gamma \}$,  the algebra of $\Gamma$-coinvariants of $B_u$.
Geometrically, the $\Gamma$-action on $B_u$ yields a $\Gamma$-action on $\Spec B_u$ and one has the equality $\Spec ((B_u)_\Gamma) = (\Spec B_u)^\Gamma$.

\begin{Proposition}\label{P:dforgetful}
Denote by $\alpha_\Gamma \colon B_u \to (B_u)_\Gamma$ the quotient map onto the coinvariant algebra.
If $\PD_A$ is representable over $B_u$, 
then  $\PD_{A}^{\Gamma}$ is representable over $(B_u)_{\Gamma}$ and  $\PD_A(\alpha_\Gamma)(\A_u, \iota_u)$ is a universal element of $\PD_A^\Gamma$. 
\end{Proposition}
\begin{proof}
Let $(\A,\iota)$ be a $\Gamma$-stable graded Poisson deformation.
By representability of $\PD_A$, there exists a unique map $\beta \in \Hom_{\GrAlg}(B_u, B)$ such that $(\A, \iota) = \PD_A(\beta)(\A_u, \iota_u)$.
We claim that $\beta = \bar \beta \circ \alpha_\Gamma$ for a unique $\bar \beta \in \Hom_{\GrAlg}((B_u)_\Gamma, B)$.
Let $(\alpha_g)_{g \in \Gamma}$ be the collection of morphisms introduced in Remark \ref{R:univisequiv}.
Then: $$\PD_A(\beta \circ \alpha_g)(\A_u, \iota_u)  = \PD(\beta)(\A_u, g \circ \iota_u) = (\A, g \circ \iota) = (\A, \iota) = \PD(\beta)(\A_u, \iota_u)$$  where in the penultimate equality we used $\Gamma$-stability of $(\A, \iota)$. 
The claim follows from uniqueness of the map $\beta$ and the universal property of the coinvariant algebra. 
We conclude that $(\A, \iota) = \PD_A^\Gamma (\bar \beta) \PD_A (\alpha_\Gamma)(\A_u , \iota_u)$.
\end{proof}

\subsection{Automorphisms of $A$ and of its quantizations} \label{ss_autqnt}
We continue to work with a fixed reductive  subgroup $\Gamma \leq \G$ as in \S \ref{ss_autdef}. The current section can be considered the quantum counterpart of the previous section.
For $B \in \SFAlg$, we have an action of $\Gamma$ on $\Qnt_A(B)$: for $(\Q, \iota) \in \Qnt_A(B)$ and $g \in \Gamma$, we define the $g$-twisted quantization $^g(\Q, \iota) = (\Q, g \circ \iota )$.

A quantum version of Remark \ref{R:univisequiv} holds. In particular if $\Qnt_A$ is representable then $\Gamma$ admits a right action on the universal base $B_u$ which we again  denote
\begin{eqnarray}
\label{e:universalQaction}
\begin{array}{rcl}
\Gamma & \longrightarrow \Aut_{\SFAlg}(B_u),
& g \longmapsto \alpha_g.
\end{array}
\end{eqnarray}

\begin{Definition}\label{D:gammastableqnt}
Let $B \in \SFAlg$.
A \emph{$\Gamma$-stable filtered quantization of $A$ over $B$} is an element $(\Q, \iota) \in \Qnt_A(B)$ such that 
$^g{(\Q, \iota)} = (\Q, \iota)$ for all $g \in \Gamma$.

For $B \in \SFAlg$, let $\Qnt_A(B)^\Gamma$ denote the subset in $\Qnt_A(B)$ consisting of isoclasses of $\Gamma$-stable filtered quantizations of $A$ over $B$.
Then for $B' \in \SFAlg$ and $\beta \in \Hom_{\SFAlg}(B, B')$,  we have $\Qnt_A(\beta) (\Qnt_A(B)^\Gamma) \subset \Qnt_A(B')^\Gamma$.
Hence we can define \emph{the functor of $\Gamma$-stable filtered quantizations of $A$}: $$\Qnt_A^\Gamma \colon \SFAlg \to \Sets$$ associating to  $B \in \SFAlg$ the set $\Qnt^\Gamma_A(B) := \Qnt_A(B)^\Gamma$ and operating on morphisms as per Defintion \ref{D:qntfunct} (ii).
\end{Definition}

Let $B \in \SFAlg$ and $(\Q, \iota) \in \Qnt_A(B)$.
Denote by $\Aut_{B-\FAlg}(\Q)$ the set of $B$-linear filtered automorphisms of $\Q$: for $\gamma \in \Aut_{B-\FAlg}(\Q)$, we have $\gr \gamma \in \PAut_{\gr B-\GrAlg}(\gr \Q)$.
In view of this, for any $(\Q, \iota) \in \Qnt_A(B)$, we define:
\begin{equation}\label{eq:qgammatilde}
\widetilde{\Gamma}_q(\Q, \iota) := \{\gamma \in \Aut_{B-\FAlg}(\Q) \mid \gr \gamma \in \widetilde{\Gamma}(\gr \Q, \iota) \}.
\end{equation}

\begin{Remark}\label{R:algebraic}
The map $\gr\colon \Aut_{B-\FAlg}(\Q)\to\PAut_{\gr B-\GrAlg}(\gr \Q)$ is a  group morphism by functoriality of $\gr$.
 By choosing $N$ sufficiently large we can (and shall) identify $\Aut_{B-\FAlg}(\Q)$ and $\PAut_{\gr B-\GrAlg}(\gr \Q)$ with closed subgroups of 
$\GL(\Q_N)$ and $\GL(\bigoplus_{j=0}^N(\gr \Q)_j)$ respectively.
Then, for any splitting of the vector space $\bigcup_{j=0}^N\Q_j=\bigoplus_{j=0}^NR_j$ with $R_0=\Q_0$ and  $R_j$ a complement of $R_{j-1}$ in $\Q_j$, the group  $\Aut_{B-\FAlg}(\Q)$ is identified with a subgroup of block upper triangular matrices. 
In addition, any such splitting is compatible with the grading of the vector space $\bigoplus_{j=0}^N(\gr \Q)_j$, and with respect to these splittings, $\gr$ is identified with the algebraic morphism mapping a block upper triangular matrix to its block diagonal part.
Hence $\gr\colon \Aut_{B-\FAlg}(\Q)\to\PAut_{\gr B-\GrAlg}(\gr \Q)$ is a morphism of algebraic groups with unipotent kernel.
\end{Remark}

\begin{Lemma} \label{L:semidirect_q}
Let $B \in \SFAlg$, $(\Q,\iota) \in \Qnt_A^\Gamma(B)$ and  $\widetilde{\Gamma}_q := \widetilde{\Gamma}_q(\Q, \iota)$, and let $p$ be as in \eqref{eq:p} for $(\gr\Q,\iota)$.
Then $\widetilde{\Gamma}_q$ is a linear algebraic group with Levi decomposition  $\widetilde{\Gamma}_q= \Gamma \ltimes U_q$ where $U_q=\ker(p \circ \gr)$  is the automorphism group of the filtered quantization $(\Q,\iota)$ over $B$. 
\end{Lemma}
\begin{proof}
Set $\widetilde{\Gamma} := \widetilde{\Gamma}(\gr \Q, \iota)$. 
Observe that $\widetilde{\Gamma}_q = \gr^{-1}(\widetilde{\Gamma})$ is a closed (hence algebraic) subgroup of $\Aut_{B-\FAlg}(\Q)$, by Remark \ref{R:algebraic}.
We know that $p$ from \eqref{eq:p} restricts to a surjective algebraic group morphism $p \colon \widetilde{\Gamma} \to \Gamma$.
By Remark \ref{R:algebraic}, from which we adopt notation, $p \circ \gr \colon \widetilde{\Gamma}_q \to \Gamma$ is  a morphism of algebraic groups. It is surjective because we assumed $(\Q,\iota) \in \Qnt_A^\Gamma(B)$.
We set $U_q = \ker(p \circ \gr)=\{g\in \widetilde{\Gamma}(\gr \Q, \iota)~|~\gr(g)\in U\}$.  By construction, it is  the automorphism group of the filtered quantization $(\Q,\iota)$ over $B$. We have a short exact sequence of algebraic groups $1 \to \ker \gr \to U_q \to \gr(U_q) \to 1$ where $\gr(U_q)\leq U$ and $\ker \gr$ are unipotent. By Mostow's Theorem, we conclude that $U_q$ is unipotent.   As a byproduct, we also derive the remaining assertions. 
\end{proof}


\begin{Definition}\label{D:equivstr_q}
Let $B \in \SFAlg$ and $(\Q, \iota) \in \Qnt_A(B)$.
Set $\widetilde{\Gamma}_q := \widetilde{\Gamma}_q(\Q, \iota)$ as in \S \ref{ss_autqnt}.
A \emph{$\Gamma$-structure on $(\Q, \iota)$} is a group morphism $s \colon \Gamma \to \widetilde{\Gamma}_q$ such that $p \circ \gr \circ s = \id_\Gamma$.
\end{Definition}

\begin{Definition} \label{def_gammaq}
Let $B \in \SFAlg$. A \emph{filtered $\Gamma$-quantization of $A$ over $B$} is a triple $(\Q, \iota, s)$
such that
\begin{enumerate}
\item[(i)] $(\Q, \iota) \in \Qnt_A(B)$;
\item[(ii)] $s$ is a $\Gamma$-structure on $(\Q, \iota)$.
\end{enumerate} 
\end{Definition}

We say that two filtered  $\Gamma$-quantizations $(\Q_1, \iota_1, s_1)$ and $(\Q_2, \iota_2, s_2)$ of $A$ over $B$ are \emph{isomorphic} if  there exists a $B$-linear filtered algebra isomorphism $\phi \colon \Q_1 \to \Q_2$ 
such that \eqref{e:iotadiagram} commutes for $(\gr\Q_1,\iota_1)$ and $(\gr \Q_2,\iota_2)$ and $\phi \circ s_1(g) = s_2(g) \circ \phi$ for all $g \in \Gamma$.

\begin{Definition} \label{def_gammaQfunctor}
The {\em functor of  filtered $\Gamma$-quantizations of $A$} is defined as follows:
$$\Qnt_{A,\Gamma} \colon \SFAlg \to \Sets$$
\begin{enumerate}[label=(\roman*)]
\item for $B \in \SFAlg$, the set $\Qnt_{A,\Gamma}(B)$ consists of isoclasses of filtered $\Gamma$-quantizations $(\Q, \iota, s)$ over $B$;
\item for $\beta \in \Hom_{\SFAlg}(B_1, B_2)$, the morphism $\Qnt_A(\beta)$ maps $(\Q_1, \iota_1, s_1) \in \Qnt_{A,\Gamma}(B_1)$ to $(\Q_2, \iota_2, s_2) \in \Qnt_{A, \Gamma}(B_2)$ where $(\Q_2,\iota_2) = \Qnt_A(\beta)(\Q_1, \iota_1)$ and for $g \in \Gamma$ we set $s_2(g) = s_1(g) \otimes \id$.
\end{enumerate}
\end{Definition}

The notion of an optimal quantization theory (OQT) from Definition~\ref{D:oc} can be upgraded to the equivariant setting.
\begin{Definition} \label{def_gammaQfunctor}
We say that $A$ admits an \emph{optimal $\Gamma$-quantization theory} (O$\Gamma$QT) if the following conditions are satisfied:
\begin{enumerate}[label=(\roman*)]
\item The functor $\Qnt_{A, \Gamma}$ is representable over some $B_{u, \Gamma} \in \SFAlg$;
\item The functor $\PD_{A, \Gamma}$ is representable over $\gr B_{u,\Gamma} \in \GrAlg$;
\item There exist a representation $\eta$ of $\Qnt_{A,\Gamma}$ over $B_{u,\Gamma}$, resp. $\zeta$ of $\PD_{A,\Gamma}$ over $\gr B_{u,\Gamma}$, such that the following diagram of natural transformations commutes:
\begin{center}
\begin{tikzcd}
\Hom_{\SFAlg}(B_{u,\Gamma}, -)   \arrow[d,,"\graded"']  \arrow[r, "\eta"] & \Qnt_{A, \Gamma}(-) \arrow[d,,"\graded"] \\
  \Hom_{\GrAlg}(\gr (B_{u,\Gamma}), \gr -) \arrow[r, "\zeta"] &  \PD_{A,\Gamma} \gr(-)
\end{tikzcd}
\end{center}
\end{enumerate}
\end{Definition}

We define the forgetful natural transformation
$\Forg \colon \Qnt_{A, \Gamma} \to \Qnt_A$ as follows: for $B \in \SFAlg$, the map $\Forg_B \colon \Qnt_{A, \Gamma}(B) \to \Qnt_A(B)$ is given by $(\Q, \iota, s) \mapsto (\Q, \iota)$.
We have a quantum analogue of Theorem 	\ref{T:dforgetful}.
\begin{Theorem} \label{T:qforgetful}
The natural transformation $\Forg \colon \Qnt_{A,\Gamma} \to \Qnt_A$ factors through $\Qnt_A^\Gamma$
and $\Forg$ defines a natural isomorphism between the functors $\Qnt_{A, \Gamma}$ and $\Qnt_A^\Gamma$.
\end{Theorem}
\begin{proof}
This is a quantum version of Theorem \ref{T:dforgetful}: the proof is analogous and uses Lemma \ref{L:semidirect_q}.
\end{proof}

As a consequence,  we identify the functors $\Qnt_{A,\Gamma}$ and $\Qnt_A^\Gamma$; in particular, for $B \in \SFAlg$, we omit the third entry (i.e. the $\Gamma$-structure $s$) when writing an element of $\Qnt_{A,\Gamma}(B)$.

The following is a quantum version of Proposition \ref{P:dforgetful}, and its proof is similar. Denote by $\alpha_\Gamma \colon B_u \to (B_u)_\Gamma$ the quotient map onto the coinvariant algebra with respect to the right $\Gamma$-action defined in \eqref{e:universalQaction}.


\begin{Proposition}\label{P:qforgetful}
If $\Qnt_A$ is representable over $B_u$, 
then  $\Qnt_A^{\Gamma}$ is representable over $(B_u)_{\Gamma}$ and  $\Qnt_A(\alpha_\Gamma)(\Q_u, \iota_u)$ is a universal element of $\Qnt_A^\Gamma$. \hfill \qed
\end{Proposition}

Recall from Section~\ref{ss:interplay} that there is a natural transformation $\graded : \Qnt_A \to \PD_A \circ\gr$. 

\begin{Lemma} \label{L:gammainterplay}
The natural transformation $\graded$ induces a natural transformation of functors $\Qnt_{A,\Gamma} \to \PD_{A,\Gamma} \circ \gr$.

Moreover, if $A$ admits an OQT the following conditions are verified:
\begin{enumerate}[label=(\arabic*)]
\item Let $B \in \SFAlg$ and $(\Q, \iota) \in \Qnt_{A,\Gamma}(B)$ and suppose $(\gr \Q, \iota) \in  \PD_{A,\Gamma}(\gr B)$ is a universal element of $\PD_{A, \Gamma}$.
Then $(\Q, \iota)$ is a universal element of $\Qnt_{A,\Gamma}$.
\item $A$ admits an O$\Gamma$\!QT with $B_{u,\Gamma}=(B_u)_{\Gamma}$, and the restrictions of the representations from the OQT assumption.
\end{enumerate}
\end{Lemma}
\begin{proof}The first assertion follows from the definitions.
For the remaining part of the Lemma,  (1) can be proven by mimicking the argument for Lemma \ref{L:interplay}. For part (3) we simply combine Theorem~\ref{T:dforgetful}, Proposition~\ref{P:dforgetful}, Theorem~\ref{T:qforgetful}, Proposition~\ref{P:qforgetful}. Commutativity of the diagram of Definition \ref{def_gammaQfunctor} (iii)  is a consequence of Propositions \ref{P:dforgetful} and \ref{P:qforgetful} under the OQT conditions for $A$.
\end{proof}

\begin{Remark}\label{R:ocgamma}
Assume $A$ admits an OQT and retain notation from Remark \ref{R:oc}. In particular the universal bases are $B_u$ and $C_u$, and $(\A_u, \iota_\A)$ and $(\Q_u, \iota_\Q)$ are universal elements of $\PD_A$ and $\Qnt_A$ (respectively) satisfying the compatibility condition (iii) of an OQT. In Remark~\ref{R:univisequiv} and \eqref{e:universalQaction} we noted that there are right $\Gamma$-actions on $B_u$ and $C_u$.
Let $\beta_\Gamma \colon B_u \to (B_u)_\Gamma$ and $\alpha_\Gamma \colon C_u \to (C_u)_\Gamma$ denote the quotient maps to the coinvariant algebras, then $\Qnt_A(\beta_\Gamma)(\Q_u, \iota_\Q)$, resp. $\PD_A(\alpha_\Gamma)(\A, \iota_\A)$ is a universal element of $\Qnt_{A,\Gamma}$, resp. of $\PD_{A,\Gamma}$.
Then (1) of Lemma \ref{L:gammainterplay} is equivalent to saying that there exists a unique isomorphism $\phi_\Gamma \in \Hom_{\GrAlg}((C_u)_\Gamma, (\gr B_u)_\Gamma)$ such that $\PD_{A, \Gamma}(\phi_\Gamma)$ maps $\PD_A(\alpha_\Gamma)(\A, \iota_\A)$ to $\gr \Qnt_A(\beta_\Gamma)(\Q, \iota_\Q)$. In other words, $\phi_\Gamma$ makes the following diagram commute:
$$
\begin{tikzcd}
C_u \arrow[rr, "\phi"] \arrow[d, "\alpha_\Gamma"]  && \gr B_u \arrow[d, "\gr \beta_\Gamma"] \\
(C_u)_\Gamma \arrow[r, "\phi_\Gamma"] & (\gr B_u)_\Gamma \arrow[r,"\sim"] &\gr((B_u)_\Gamma)
\end{tikzcd}$$
\end{Remark}

\begin{Lemma}\label{L:interplay2}
Assume $A$ admits an OQT and retain notation from Remark \ref{R:oc}.
Let $(\Q, \iota) \in \Qnt_A(B_u)$ satisfy $(\gr \Q_u, \iota_\Q) = (\gr \Q, \iota)$.
Then there exists a unique unipotent automorphism $\nu \in \Aut_{\SFAlg}(B_u)$ such that $\Qnt_A(\nu)(\Q_u, \iota_\Q) = (\Q, \iota)$.
\end{Lemma}
\begin{proof}
Under the assumptions, there exists a unique $\nu \in \Hom_{\SFAlg}(B_u, B_u)$ such that $\Qnt_A(\nu)(\Q_u, \iota_\Q) = (\Q, \iota)$.
By Lemma \ref{L:interplay}, $(\Q, \iota)$ is universal, so $\nu \in \Aut_{\SFAlg}(B_u)$.
Moreover, it satisfies the relation $(\gr \nu) \circ \phi = \phi \in \Hom_{\GrAlg}(C_u, \gr B_u)$.
Since $\phi$ is an isomorphism, this implies $\gr \nu = \id_{\gr B_u}$.
To prove that $\nu$ is unipotent, we observe that $\Aut_{\SFAlg}(B_u)$ and $\Aut_{\GrAlg}(\gr B_u)$ are algebraic groups and that $\gr \colon \Aut_{\SFAlg}(B_u) \to \Aut_{\GrAlg}(\gr B_u)$ is a morphism of algebraic groups: one can use the same argument in Remark \ref{R:algebraic}.
\end{proof}

Along the same lines one can prove the following $\Gamma$-equivariant version of the result.
\begin{Lemma}\label{L:gammainterplay2}
Assume $A$ admits an OQT and retain notation from Remark \ref{R:ocgamma}. 
Let $(\Q, \iota) \in \Qnt_{A,\Gamma} ((B_u)_\Gamma)$ satisfy $\graded_{(B_u)_\Gamma} \Qnt_A(\beta_\Gamma)(\Q, \iota_\Q) = (\gr \Q, \iota)$.
Then there exists a unique unipotent automorphism $\nu \in \Aut_{\SFAlg}((B_u)_\Gamma)$ such that $\Qnt_{A,\Gamma}(\nu)(\Q_u, \iota_\Q) = (\Q, \iota)$. \hfill \qed
\end{Lemma}

\subsection{Conical symplectic singularities and their deformations}
\label{ss:conicalsingulairities}
In this section we apply all of the above results to an important class of conical Poisson varieties, known as symplectic singularities. We say that an affine Poisson variety $X$ is {\it conical} if $\C[X]$ is a graded Poisson algebra in degree $-n$ for some $n\in \Nbb$. Now let $X$ be a normal conical  Poisson variety, with Poisson bivector $\varpi$, and suppose that the restriction $\varpi|_{X^\reg}$ to the regular locus is non-degenerate, i.e. $(X^\reg, \varpi)$ is a symplectic variety. Following \cite{Be} we say that $X$ {\it is a symplectic singularity} if there is a projective resolution of singularities $\rho : \tX \to X$ such that the symplectic form $\rho^*\pi$ on $\rho^{-1}(X^\reg)$ extends to a regular $2$-form on $\tX$. 
This property does not depend on which resolution you choose \cite[\textsection 2.1]{Ka}.

The following theorem combines results of Losev and Namikawa; see \cite{Lo, Na1, Na2}.
\begin{Theorem}
\label{T:conicsymphaveinitials}
Let $X$ be a conical symplectic singularity.  
Then $A = \C[X]$ admits an OQT,  so Lemmas~\ref{L:interplay} and  \ref{L:interplay2}  apply. 
\end{Theorem}
\begin{Remark}
The universal base of the functor $\PD_A$ and $\Qnt_A$ is the coordinate ring on a graded vector space defined to be the cohomology $H^2(\tilde X^\reg, \C)$ where $\tilde X^\reg$ is the smooth locus of a $\mathbb{Q}$-factorial terminalization of the conical symplectic singularity $X$. Since these technical details lie beyond the requirements of the current paper we refer the reader to \cite{Na1, Lo} for information.
\end{Remark}

Now let $X$ be a conical symplectic singularity with $\Gamma$ a reductive group of $\C^\times$-equivariant  Poisson automorphisms.
Then Propositions~\ref{P:dforgetful}, \ref{P:qforgetful} and Lemmas \ref{L:gammainterplay}, \ref{L:gammainterplay2}  apply. This completes the proof of Theorem~\ref{T:firstmain}.\\

For the sake of completeness we make explicit the excellent properties of $\Gamma$-quantization theory for $X$.
\begin{Theorem} Let $X$ be a conical symplectic singularity.
If $\Gamma$ is any reductive group of $\C^\times$-equivariant Poisson automorphisms of $X$ then $\C[X]$ admits an O$\Gamma$\!QT.
\end{Theorem}
\begin{proof}
Apply Lemma~\ref{L:gammainterplay} and Theorem~\ref{T:conicsymphaveinitials}.
\end{proof}

%
\section{Nilpotent Slodowy slices and their universal quantizations}
\label{S:nilpotentslodowy}
Throughout this section we use the following notation:
\begin{itemize}
\setlength{\itemsep}{4pt}
\item $G$ is a complex simple, simply-connected algebraic group;
\item $\g = \Lie(G)$ the Lie algebra;
\item $\kappa : \g \to \g^*$ is the $G$-equivariant isomorphism induced by the Killing form;
\item $\N(\g^*) = \kappa(\N(\g))$ is the nullcone of $\g^*$;
\item $e \in \N(\g)$ is a nilpotent element with $\chi := \kappa(e)$;
\item $(e,h,f)$ is an $\sl_2$-triple;
\item $\h$ is a maximal toral subalgebra containing $h$, $\Delta$ is the Dynkin diagram and $W$ the Weyl group;
\item for a given choice of a base for the root system $\Phi$ of $\g$,  $\rho$ is half-sum of the positive roots;
\item $\lambda : \C^\times \to G$ is a cocharacter with $d_1 \lambda(t) = th$;
\item $\Ss_\chi  = \chi + \kappa(\g^f) \subseteq \g^*$ is the Slodowy slice.
\end{itemize}

\subsection{Hamiltonian reduction}
\label{ss:Hamiltonianreduction}

Before we introduce the main objects of study of this section, we record some preliminaries on Hamiltonian reduction. In fact the ideas introduced here serve to generalise the well-known formalism of Hamiltonian reduction exhibited in the literature, and so we include a brief proof for the reader's convenience.

Let $N$ be a complex affine group and let $X$ be a complex affine Poisson variety. An $N$-action on $X$ is said to be {\it Hamiltonian} if there is an $N$-equivariant Poisson homomorphism $\mu^* : \C[\n^*] \to \C[X]$ which satisfies
\begin{eqnarray}
\label{e:comoment}
d_1 \rho(x)= \{\mu^*(x), \cdot\}
\end{eqnarray}
where $\rho : N \to \Aut \C[X]$ is the (locally finite) representation of $N$ on $\C[X]$ and $d_1\rho : \n \to \Der\C[X]$ is the differential. The map $\mu^*$ is called the {\it comoment map} whilst the induced morphism $\mu : X\to \n^*$ is the {\it moment map}. 

For any $N$-stable affine subvariety $Y \subseteq \n^*$  we define the {\it Hamiltonian reduction}
\begin{eqnarray}
\label{e:Hamiltonianred}
\mu^{-1}(Y)/\!/ N := \Spec (\C[\mu^{-1}(Y)]^N).
\end{eqnarray}
We note that, in general, the algebra $\C[\mu^{-1}(Y)]^N$ is not finitely generated and so the Hamiltonian reduction is a scheme but not necessarily an algebraic variety. Nonetheless the $N$-stable assumption on $Y$ ensures that Hamiltonian reduction inherits a Poisson structure.

\begin{Lemma}
\label{L:HamRed}
$\mu^{-1}(Y)/\!/N$ inherits a Poisson structure from $X$. The bracket is given by lifting functions to $\C[X]$, taking the Poisson bracket and restricting to $\mu^{-1}(Y)/\!/N$.
\end{Lemma}
\begin{proof}
Let $J \subseteq \C[\n^*]$ be the defining ideal of $Y$ with generators $f_1,...,f_n$. Then the defining ideal of the scheme-theoretic fibre $\mu^{-1}(Y)$ is $I = (\mu^*(f_i) \mid i =1,...,n) \unlhd \C[X]$.

Since $Y$ is $N$-stable and $\mu$ is $N$-equivariant it follows that $\mu^{-1}(Y)$ is $N$-stable, and so $N$ acts on $\C[\mu^{-1}(Y)]$. Since the action is locally finite we have $\C[\mu^{-1}(Y)]^N = \C[\mu^{-1}(Y)]^\n$.

In order to prove the claim it suffices to show that for $g_1 + I, g_2 + I \in \C[\mu^{-1}(Y)]^\n$ the bracket
$$\{g_1 + I, g_2 + I\} := \{g_1, g_2\} + I$$
is well-defined. This will follow if we show that every $g_1, g_2 \in \C[X]$ satisfying $g_i + I \in \C[\mu^{-1}(Y)]^\n$ actually lie in the Poisson idealiser of $I$, which is defined to be the subalgebra consisting of elements $g\in \C[X]$ such that $\{g, I\} \subseteq I$.

Let $g \in \C[X]$ be such that $g + I \in \C[\mu^{-1}(Y)]^\n$. The $\n$-invariance of $g + I$ can be rewritten as $d_1\rho(\n) (g + I) \subseteq I$. Since $I$ is $\n$-stable this is equivalent to $d_1 \rho(\n) g \subseteq I$. Now formula \eqref{e:comoment} implies that $\{\mu^*(\n), g\} \subseteq I$. Note that $\mu^*(f_i)$ lies in the symmetric algebra $S(\mu^*(\n)) \subseteq \C[X]$ and so applying the Leibniz identity we obtain $\{\mu^*(f_i), g\} \subseteq I$. This shows that $\{g, I\} \subseteq I$, and the proof is complete.
\end{proof}

\subsection{Poisson structures on Slodowy slices}
\label{ss:Poissonslices}

We begin by explaining how $\Ss_\chi$ is naturally equipped with a conical Poisson structure. This structure can be understood in two different ways: either as the transverse Poisson structure to $\g^*$ at $\chi$ as in \cite[\textsection 2.3]{DKV}, or alternatively via Hamiltonian reduction similar to \cite{GG}, as we now explain.

The torus $\lambda(\C^\times)$ induces a $\Z$-grading on $\g^*$ via $\g^*(i) = \{\xi \in \g^* \mid \lambda(t) \cdot \xi = t^i \xi\}$, where $\lambda(t) \cdot\xi$ denotes the coadjoint action. Using the representation theory of $\sl_2$, we have $\kappa(\g^f) \subseteq \bigoplus_{i \le 0} \g^*(i)$. Therefore the cocharacter $\C^\times \to \GL(\g^*)$ given by $t \mapsto t^{-2} \lambda(t)$ induces a contracting $\C^\times$-action on $\Ss_\chi$ with negative weights. This action defines a grading on both $\C[\g^*]$ and $\C[\Ss_\chi]$, known as the {\it Kazhdan grading}. It is readily seen that the Poisson bracket on $\C[\g^*]$ lies in degree $-2$, and that the grading on $\C[\Ss_\chi]$ is non-negative.

Thanks to the representation theory of $\sl_2$ we have an isomorphism $\ad(e) : \g(-1) \isoto \g(1)$, and it follows that the form $\omega \colon x ,y \mapsto \chi[x,y]$ on $\g(-1)$ is symplectic. We pick an isotropic subspace $\ell \subseteq \g(-1)$. We let $\ell^{\perp_\omega} \subseteq\g(-1)$ be the annihilator of $\ell$ with respect to $\omega$ and let $N_\ell \subseteq G$ be a unipotent algebraic group such that $\n_\ell = \Lie(N_\ell)$, where
\begin{eqnarray*}
& & \n_\ell = \ell^{\perp_\omega} \oplus \bigoplus_{i < -1} \g(i);\\
& & \m_\ell = \ell \oplus \bigoplus_{i < -1} \g(i).
\end{eqnarray*}

The group $N_\ell$ acts by Poisson automorphisms on $\g^*$ and $\mu_\ell^*$, and the restriction map $\mu_\ell : \g^* \to \n_\ell^*$ is $N_\ell$-equivariant. In fact $\mu_\ell$ is a comoment map for the Hamiltonian action of $N_\ell$ on $\g^*$, and this places us in the context for Hamiltonian reduction; see \cite[\textsection 5.4.4]{LPV}.

 Consider the set
\begin{eqnarray}
\label{e:introY}
Y_\ell := \chi|_{\n_\ell} + \Ann_{\n^*_\ell}(\m_\ell).
\end{eqnarray}
Thanks to \cite[Lemma 2.1]{GG} the coadjoint action gives an isomorphism
\begin{eqnarray}
\label{e:adjointisomorphism}
N_\ell \times \Ss_\chi \isoto \mu_\ell^{-1}(Y_\ell)  = \chi + \Ann_{\g^*}(\m_\ell).
\end{eqnarray}
Therefore $N_\ell$ acts freely on $\mu_\ell^{-1}(Y_\ell)$ and the slice $\Ss_\chi$ parameterises $N_\ell$-orbits in $\mu^{-1}_\ell(Y_\ell)$. It follows that there is a natural isomorphism of Kazhdan graded algebras
\begin{eqnarray}
\label{e:sliceclaim}
\C[\mu_\ell^{-1}(Y_\ell)]^{\ad(\n_\ell)} = \C[\mu_\ell^{-1}(Y_\ell)]^{N_\ell} \isoto \C[\Ss_\chi].
\end{eqnarray}

We write $I_\chi \subseteq \C[\g^*]$ for the defining ideal of $\mu_\ell^{-1}(Y_\ell)$, which is generated by $x - \chi(x)$ with $x\in \m_\ell$. By Lemma~\ref{L:HamRed} there is a Poisson structure on $\C[\mu_\ell^{-1}(Y_\ell)]^{\ad(\n_\ell)}$ given by $\{f + I_\chi, g + I_\chi\} := \{f,g\} + I_\chi$ for $f + I_\chi, g + I_\chi \in \C[\mu_\ell^{-1}(Y_\ell)]^{\ad(\n_\ell)}$. This Poisson structure is transferred from $\C[\mu_\ell^{-1}(Y_\ell)]^{\ad(\n_\ell)}$ to $\C[\Ss_\chi]$ via the isomorphism \eqref{e:sliceclaim}.

Observe that $\mu_\ell^{-1}(Y_\ell) \into \mu_0^{-1}(Y_0)$ and so $\C[\mu_0^{-1}(Y_0)] \onto \C[\mu_\ell^{-1}(Y_\ell)]$. Using the fact that $\Ss_\chi \subseteq \mu_\ell^{-1}(Y_\ell)$, along with \eqref{e:adjointisomorphism} we see that
 $\C[\mu_0^{-1}(Y_0)]^{\ad(\n_0)} \to \C[\mu_\ell^{-1}(Y_\ell)]^{\ad(\n_\ell)}$ which is an isomorphism of Poisson algebras because both algebras are isomorphic to $\C[\Ss_\chi]$ as Kazhdan graded algebras, by \eqref{e:sliceclaim}. Hence the Poisson structure which we have placed on $\Ss_\chi$ does not depend on $\ell$.

Now consider the Poisson subvariety $\Ss_{\chi, \N} := \Ss_\chi \cap \N(\g^*)$, known as the {\it nilpotent Slodowy variety}.
\begin{Lemma}
\label{L:Slodowyconicsymp}
$\Ss_{\chi, \N}$ is a conical symplectic singularity.
\end{Lemma}
\begin{proof}
Thanks to \cite[\textsection 5]{PrST} the fibres of the restriction of the adjoint quotient map $\Ss_{\chi} \to \h/W$ are irreducible and normal, in particular the zero fibre $\Ss_{\chi, \N}$ is normal. Since the Hamiltonian reduction of a smooth symplectic variety is symplectic \cite[Corollary~6.16]{LPV}, the symplectic leaves of $\Ss_\chi$ are the irreducible components of the intersections of the leaves of $\g^*$ with $\Ss_\chi$, i.e. the components of the intersections $G\cdot \xi \cap \Ss_\chi$ with $\xi \in \g^*$.
It follows that $\Ss_{\chi, \N}$ contains a dense symplectic leaf $\Ss_{\chi, \N}^\reg := \Ss_\chi \cap \O_{\reg}$ corresponding to the regular coadjoint orbit. This shows that $\Ss_{\chi, \N}$ is a Poisson variety of full rank, such that $\C[\Ss_{\chi, \N}]$ is positively graded  with bracket in degree $-2$. 
It follows from \cite[Proposition~2.1.2]{Gi} that $\Ss_{\chi, \N}$ admits a symplectic resolution, which completes the proof of the current lemma.
\end{proof}

\subsection{Finite $W$-algebras}
\label{ss:finiteWalgebras}
Let $\m_{\ell, \chi} = \{x - \chi(x) \mid x\in \m_\ell\} \subseteq U(\g)$ and consider the left ideal $J_\chi := U(\g) \m_{\ell, \chi}$. A short calculation will confirm that $\ad(\n_\ell)$ preserves $J_\chi$ and the $\ad(\n_\ell)$-invariants in the left $U(\g)$-module $Q := U(\g) / J_\chi$ inherit an algebra structure from $U(\g)$. The algebra $U(\g,e) := Q^{\ad(\n_\ell)}$ is known as the {\it finite $W$-algebra}.

Define a filtration $U(\g) = \bigcup_{i \in \Z} F_i U(\g)$ by placing $\g(i)$ in  $F_{i + 2}U(\g)$; we warn the reader that $F_i U(\g) \neq 0$ for all $i\in \Z$, contrary to the conventions of the rest of this paper. This descends to a non-negative filtration on both $Q$ and $U(\g,e)$ known as the {\it Kazhdan filtration}. The associated graded algebra is $\gr U(\g) \simeq \C[\g^*]$ and under this isomorphism we have an identification $\gr J_\chi = I_\chi$. Since $U(\g)$ is almost commutative with respect to this filtration (of degree $-2$) so too is $U(\g,e)$. Therefore $\gr U(\g,e)$ is equipped with a Poisson structure in the usual manner.

By \cite[Proposition~5.2]{GG} the natural inclusion $\gr U(\g,e)\subseteq (\C[\g^*] / I_\chi)^{\ad(\n_\ell)} \simeq \C[\Ss_\chi]$ is an equality, and it is not hard to check that the Poisson structure on $\gr U(\g,e)$ arising from the noncommutative multiplication coincides with the structure coming from Poisson reduction of $\C[\g^*]$. Thus $U(\g,e)$ is a filtered algebra quantizing the Kazhdan graded Poisson algebra $\C[\Ss_\chi]$.

\subsection{Casimirs on the Slodowy slice and the centre of the finite $W$-algebra}
\label{ss:casimirsandcentres}
We have chosen our maximal toral subalgebra $\h\subseteq \g$ so that $h\in \h$, and $W$ denotes the Weyl group. Therefore $\C[\h^*] \into \C[\g^*]$ is a Kazhdan graded subalgebra with $\h \subseteq \C[\h^*]$ concentrated in degree 2, and $W$ acts by graded automorphisms. The $\rho$-shifted invariants are denoted $\C[\h^*]^{W_\bullet}$ as usual.

The Poisson centre of $\C[\g^*]$ is $Z\C[\g^*] = \C[\g^*]^G$ and the centre of $U(\g)$ is $Z(\g) = U(\g)^G$. These algebras are well-understood by the Chevalley restriction theorem and the Harish-Chandra restriction theorem. Consider the natural projection maps
\begin{eqnarray}
\label{e:invariantsvscentres}
\begin{array}{rcccccl} Z\C[\g^*] &=& \C[\g^*]^{G} &\to & Z\C[\Ss_\chi] & & \vspace{4pt} \\
Z(\g) &=& U(\g)^{G} &\to & Z(\g,e) &:=& ZU(\g,e)\end{array}
\end{eqnarray}

\begin{Lemma} \cite[Footnote~1]{PrJI}
\label{L:centrelemma}
The maps \eqref{e:invariantsvscentres} are isomorphisms.
$\hfill \qed$
\end{Lemma}

We have the following commutative diagram
\begin{eqnarray}
\label{e:centres}
\begin{tikzcd}[column sep=1.5em, row sep = 1.5em]
\C[\g^*] \arrow{d}{\text{res}} & \arrow[l,hook'] \C[\g^*]^G \arrow{d}{\simeq} \arrow{r}{\simeq} & \C[\h^*]^W \arrow{r}{\simeq} & \C[\h^*]^{W_\bullet} & \arrow{l}[swap]{\simeq} U(\g)^G \arrow{d}{\simeq} \arrow[r,hook] & U(\g) \arrow{d}{\text{pr}} \\
\C[\Ss_\chi] & \arrow[l, hook'] Z\C[\Ss_\chi] \arrow{rrr}{\simeq} & & & Z(\g,e) \arrow[r, hook] & Q.
\end{tikzcd}
\end{eqnarray}
The restriction map $\C[\g^*]^G \to \C[\h^*]^W$ is an isomorphism by Chevalley's restriction theorem, $U(\g)^G \to \C[\h^*]^{W_\bullet}$ is the Harish-Chandra isomorphism and the isomorphism $\C[\h^*]^W \to \C[\h^*]^{W_\bullet}$ is the shift map $x \mapsto x - \rho(x)$ which sends invariants to $\rho$-shifted invariants. The isomorphism $Z\C[\Ss_\chi] \to Z(\g,e)$ is the unique map making the diagram commute. Every algebra on the left half of \eqref{e:centres} is Kazhdan graded. Furthermore, if we grade $\C[\h^*]$ by placing $\{x - \rho(x) \mid x \in \h\}$ in degree 2 then $\C[\h^*]^{W_\bullet}$ is a graded subalgebra and $\C[\h^*]^W \to \C[\h^*]^{W_\bullet}$ is a graded homomorphism. Furthermore the isomorphism $\C[\h^*]^{W_\bullet} \to Z(\g,e)$ is strict for the Kazhdan filtration.

\subsection{The universal deformation of a nilpotent Slodowy slice}
\label{ss:universaldeformationoftheslice}
For the rest of the paper we identify $Z\C[\Ss_\chi] = \C[\h^*/W]$ and $Z(\g,e) = \C[\h^*/{W_\bullet}]$ as Kazhdan graded algebras, via \eqref{e:centres}. Since the scheme-theoretic fibres of the adjoint quotient map $\Ss_\chi \to \h^*/W$ are reduced \cite[Theorem~5.4(ii)]{PrST}, it follows from Kostant's theorem \cite[Proposition~7.13]{JaNO} that $\C[\Ss_\chi] \otimes_{\C[\h^*/W]} \C_+ \simeq \C[\Ss_{\chi, \N}]$ as graded Poisson algebras. For the rest of the section we pick a graded Poisson isomorphism
\begin{eqnarray*}
\iota : \C[\Ss_\chi] \otimes_{\C[\h^*/W]} \C_+ \to \C[\Ss_{\chi, \N}].
\end{eqnarray*}
By \cite[Corollary~7.4.1]{Slo}, the adjoint quotient map $\Ss_\chi \to \h^*/W$ is flat, which completes the proof of the next result.
\begin{Lemma}
\label{L:Walgebragraded}
Set $A = \C[\Ss_{\chi, \N}]$. Then the following hold:
\begin{enumerate}[label=(\arabic*)]
\item $(\C[\Ss_\chi], \iota) \in \PD_A(\C[\h^*/W])$;
\item $(U(\g,e),  \iota) \in \Qnt_A(\C[\h^*/{W_\bullet}])$;
\item $(\C[\Ss_\chi], \iota)$ is the associated graded deformation  of $(U(\g,e), \iota)$. $\hfill \qed$
\end{enumerate}
\end{Lemma}

Combining Theorem~\ref{T:conicsymphaveinitials} and Lemma~\ref{L:Slodowyconicsymp} we see that the graded Poisson algebra $\C[\Ss_{\chi, \N}]$ admits an OQT.
 It is natural to wonder under what circumstances the objects in Lemma~\ref{L:Walgebragraded} are universal. This question was answered comprehensively by Lehn--Namikawa--Sorger, as we recalled in the introduction to this paper. We record their result here for the reader's convenience.
\begin{Theorem} \cite[Theorem~1.2 \& 1.3]{LNS} 
\label{T:LSNtheorem}
Let $\g$ be a simple Lie algebra and $e \in \g$ be a nilpotent element with orbit $\O$.

Set $A = \C[\Ss_{\chi, \N}]$.
Then the following are equivalent:
\begin{enumerate}
\item $\PD_A$ is represented by $\C[\h^*/{W}]$ and
$(\C[\Ss_\chi], \iota)$ is a universal Poisson deformation of $A$.
\item $(\g, \O)$ does not occur in the following table.
\begin{center}

\medskip

\begin{tabular}{|c|c|c|c|c|c|c|}
\hline
Type of $\g$ & Any & {\sf BCFG} & {\sf C} & \sf{G} 

\\
\hline

Type of $\O$ & Regular & Subregular & Two Jordan blocks & dimension 8 \\ 
\hline
\end{tabular}
\end{center}
\begin{center}
\end{center}
\end{enumerate}
\end{Theorem}

 In \cite{LNS} the authors actually classified the nilpotent orbits for which the adjoint quotient $\Ss_\chi \to \h^*/W$ is the formally universal Poisson deformation. It is explained by Namikawa in \cite[\textsection 5]{Na1} that when the underlying affine Poisson variety is conical a formally universal deformation can be globalised, leading to a universal Poisson deformation in the sense of the current paper (see also \cite[\textsection 2.2]{Lo}). The regular Slodowy slice is not discussed explicitly in \cite{LNS}, however it is a classical theorem of Kostant \cite{Ko2} that $\Ss_\chi \to \h^*/W$ is an isomorphism for $\chi$ regular, and so the Poisson structure is trivial in these cases by Lemma~\ref{L:centrelemma}. 
  
 The following is one of our main results. For the proof one should combine Lemma \ref{L:interplay}, Lemma~\ref{L:Slodowyconicsymp}, Lemma~\ref{L:Walgebragraded} and Theorem~\ref{T:LSNtheorem}.

\begin{Theorem}
\label{T:Walgfilteredquant}
Set $A = \C[\Ss_{\chi, \N}]$. The following are equivalent:
\begin{enumerate}[label=(\arabic*)]
\item $\Qnt_A$ is represented by $\C[\h^*/{W_\bullet}]$ and $(U(\g,e), \iota)$ is a universal element of $\Qnt_A$;
\item the orbit of $e$ is not listed in the above table.  $\hfill \qed$
\end{enumerate}
\end{Theorem}



\begin{Remark}
\label{R:simplylacedisomorphisms}
The universal property in Theorem \ref{T:Walgfilteredquant} leads to exceptional isomorphisms with other interesting algebras arising in representation theory.
In particular, \cite[Proposition~3.17]{Lo}  shows that a universal quantization of a simple surface singularity is given by (the Namikawa--Weyl group invariants in) the rational Cherednik algebra for the Weyl group of the same Dynkin type.
By the work of Brieskorn and Slodowy we know that these surface singularities are isomorphic to subregular nilpotent Slodowy slices for simply-laced Lie algebras.
Hence the subregular simply-laced finite $W$-algebras are isomorphic to the corresponding spherical symplectic reflection algebras.
This observation also follows from Losev's Theorems 5.3.1 \& 6.2.2 of \cite{Lo12}.
\end{Remark}

\section{Deformations in the subregular case}
\label{S:subregular}

We retain the notation and assumptions of Section~\ref{S:nilpotentslodowy}. On top of this we assume henceforth that $e\in\g$ is a subregular element.

\subsection{The subregular slice and the automorphism group}
\label{ss:automorphismsandsubregular}
Consider the subgroup $C = C(e,h,f) \subseteq \Aut(\g)$ consisting of automorphisms fixing the triple. Its structure is described in \cite[\textsection 7.5]{Slo}. The action of $C$ on $\g$ descends to an action on $\Ss_\chi$.
\begin{Lemma}
\label{L:Gamma0Poisson}
$C$ acts on $\C[\Ss_\chi]$ by graded Poisson automorphisms.
\end{Lemma} 
\begin{proof}
Recall the notation $\ell, N_\ell, \mu, I_\chi, Y_\ell$ from Section \ref{ss:Poissonslices} and set $\ell = 0$ and $\mu:=\mu_0$. Since the Poisson structure on $\C[\Ss_\chi]$ is defined via the graded isomorphism \eqref{e:sliceclaim} it will suffice to show that $C$ acts by Poisson automorphisms on $\C[\mu^{-1}(Y_0)]^{N_0}$. Since $C$ preserves the graded pieces of $\g$, it stabilises both $\m_0$ and $\n_0$, and furthermore acts on $\C[\g^*]$ by automorphisms which preserve the Kazhdan grading. The defining ideal $I_\chi$ of $\mu^{-1}(Y_0)$ in $\C[\g^*]$ is generated by the Kazhdan graded vector space $\{x - \chi(x) \mid x\in \m_0\}$ and so $C$ acts by graded automorphisms on $\C[\mu^{-1}(Y_0)]^{\ad(\n_0)}$. Since $N_0$ is connected and unipotent the latter algebra coincides with $\C[\mu^{-1}(Y_0)]^{N_0}$. To see that the $C$-action on $\C[\mu^{-1}(Y_0)]^{N_0}$ is Poisson it suffices to recall that $\{f + I_\chi, g + I_\chi\} := \{f,g\} + I_\chi$ for $f + I_\chi, g + I_\chi \in \C[\mu^{-1}(Y_0)]^{N_0}$.
\end{proof}

\subsection{The equivariant universal deformation of a subregular nilpotent Slodowy slice}
\label{ss:equivariantuniversal}

Assume now $\g_0$ is not simply-laced and choose a simple Lie algebra $\g$ by determining the Dynkin type as follows:
\begin{eqnarray}
\label{e:choosingdynkintypes}
\left\{\begin{array}{cl} {\sf A_{2n-1}} & \text{if } \g_0 \text{ is of type } {\sf B_{n}}\\
																{\sf D_{n+1}} & \text{if } \g_0 \text{ is of type } {\sf C_{n}}\\
																{\sf E_{6}} & \text{if } \g_0 \text{ is of type } {\sf F_{4}}\\
																{\sf D_{4}} & \text{if } \g_0 \text{ is of type } {\sf G_{2}}.
																\end{array}\right\}.
\end{eqnarray}
In this section we consider the subregular Slodowy slice in $\g_0^*$ and so we use notation $e_0, \chi_0, \Ss_{\chi_0}$ to mirror the notation for $\g$. The nilpotent subregular Slodowy slice for $\g_0$ is denoted $\Ss_{\chi_0, \N_0}$. The following lemma appeared in \cite[Lemma~2.23]{EG}, we include here another proof for the reader's convenience.
\begin{Lemma}
\label{L:subregularPoissoniso}
The Poisson varieties $\Ss_{\chi, \N}$ and $\Ss_{\chi_0, \N_0}$ are $\C^\times$-isomorphic.
\end{Lemma}
\begin{proof}
It follows from the proofs of \cite[Theorem~8.4 \& 8.7]{Slo} that $\Ss_{\chi, \N}$ and $\Ss_{\chi_0, \N_0}$ are both $\C^\times$-isomorphic to a simple surface singularity, say $\C[\Ss_{\chi, \N}] \simeq \C[x,y]^\Gamma \simeq \C[\Ss_{\chi_0, \N_0}]$ where $\Gamma$ is a finite subgroup of $\SL_2$. Let $\{\cdot, \cdot\}$ and $\{\cdot, \cdot\}_0$ denote the Poisson structures on $\C^2/ \Gamma$ transported from $\Ss_{\chi, \N}$ and $\Ss_{\chi_0, \N_0}$ respectively. Applying the argument in the final paragraph of the proof of \cite[Proposition~9.24]{LPV} we see that $\{\cdot, \cdot\} = c\{\cdot, \cdot\}_0$ for some $c\in \C^\times$. Now apply Remark~6.19 of {\it op. cit.} to complete the proof.
\end{proof}

 Let $\Gamma_0$ be the finite subgroup of $C$ defined in \cite[p. 143]{Slo}. It is isomorphic to the group $\Aut(\Delta)$ of Dynkin diagram automorphism of $\g$ for all pairs $(\g,\g_0)$  except for $({\sf D}_4, {\sf C}_3)$ in which case $\Gamma_0$  is isomorphic to a subgroup of order $2$.  In all cases the composition $\Gamma_0\into C\into \Aut(\g) \to \Aut(\Delta)$ is injective and its image is the subgroup of $\Aut(\Delta)$ realising the Dynkin diagram $\Delta_0$ of $\g_0$ as a folding of $\Delta$.

By \cite[\S 8.7, Remark 3]{Slo} there is a morphism of deformations of the algebraic variety $\Ss_{\chi_0,\N_0}\simeq \Ss_{\chi,\N}$
 illustrated by the following diagram, where the vertical arrows are the adjoint quotient maps:
\begin{eqnarray}
\label{e:slodowy}
\begin{tikzcd}[column sep=1.5em, row sep = 1.5em]
\Ss_{\chi_0} \arrow{r}{i} \arrow{d}{\delta_0} & \Ss_{\chi}\arrow{d}{\delta} \\
\h^*_0/W_0 \arrow{r}{j} & \h^*/W.
\end{tikzcd}
\end{eqnarray} 
Observe that $\Gamma_0$ acts on $\Ss_{\chi}$ and on $\h^*/W$ and that the map $\delta$ is  $\Gamma_0$-equivariant.
By \cite[\S 8.8, Remark 4]{Slo}  the maps $(i, j)$ induce  isomorphisms of varieties
\begin{eqnarray}
\label{eq:invariants}
\h^*_0/W_0\simeq(\h^*/W)^{\Gamma_0},\quad\quad\Ss_{\chi_0}\simeq\Ss_{\chi}\times_{\h^*/W}\h^*_0/W_0
\end{eqnarray}
 where $(\h^*/W)^{\Gamma_0} \subseteq \h^*/W$ is the subscheme of $\Gamma_0$-fixed points.
 
\begin{Example}
\label{E:ABexample}
Assume $\g_0$ is of type ${\sf B}_n$, and $\g$ is of type ${\sf A}_{2n-1}$. Then $\Gamma_0$ is a cyclic group of order 2 and $\C[\h^*]=\C[x_1,\,\ldots,\,x_{2n}]/(x_1+\cdots+x_{2n})$. The only non-trivial element $\gamma$ in $\Gamma_0$ maps $x_i$ to $-x_{2n+1-i}$. Furthermore
$$\C[\h^*/W]=\left(\C[x_1,\,\ldots,\,x_{2n}]/(x_1+\cdots+x_{2n})\right)^{{\mathfrak S}_{2n}} = \C[e_2, e_3,\,\ldots,\,e_{2n}]$$ where $e_j$ is the $j$th elementary symmetric polynomial. Thus $e_j\cdot \gamma=e_j$ for $j$ even and $e_j \cdot \gamma=-e_j$ for $j$ odd and the kernel of the natural projection $\C[\h^*/W] \to \C[(\h^*/ W)^{\Gamma_0}]$ is generated by all $e_{2r+1}$ for $r=1,\,\ldots,\,n-1$. 
\end{Example}



For each piece of notation at the start of Section~\ref{S:nilpotentslodowy} we introduce the same notation for $\g_0$. For example, $\h_0 \subseteq \g_0$ is a choice of maximal toral subalgebra and $W_0$ is the corresponding Weyl group.
Applying the remarks of Section~\ref{ss:universaldeformationoftheslice} we see that we may fix graded isomorphisms
\begin{eqnarray*}
\iota_0 &\colon &\C[\Ss_{\chi_0}] \otimes_{\C[\h_0^*/{W_0}]} \C_+ \to \C[\Ss_{\chi_0, \N_0}];\\
\iota& \colon &\C[\Ss_\chi] \otimes_{\C[\h^*/{W}]} \C_+ \to \C[\Ss_{\chi, \N}].
\end{eqnarray*}
such that $\iota$ is $\Gamma_0$-equivariant.
Since the reductive group $\Gamma_0$ acts on $\C[\Ss_{\chi, \N}]$ by graded Poisson automorphisms we can consider  universal $\Gamma_0$-deformations of $\Ss_{\chi, \N}$. 
 After fixing an isomorphism as in Lemma \ref{L:subregularPoissoniso}, by an abuse of notation we will identify the graded Poisson algebra $\C[\Ss_{\chi_0, \N_0}]$ with $\C[\Ss_{\chi, \N}]$ and view $\iota_0$ as an isomorphism $\C[\Ss_{\chi_0}] \otimes_{\C[\h_0^*/{W_0}]} \C_+ \to \C[\Ss_{\chi, \N}]$.

Before we proceed to the main result of this section, we prove an auxiliary  Lemma containing some general observations regarding graded homomorphisms.

\begin{Lemma}\label{L:auxiliary}Let $V = \bigoplus_{i=1}^n V_i$ and $U = \bigoplus_{i=1}^n U_i$ be finite dimensional positively graded vector spaces, with $V_i$ and $U_i$ in degree $i$,  with possibly $V_i = 0$ or $U_i = 0$ for some $1\le i \le n$.  Let $\tau : S(V) \to S(U)$ be a graded algebra homomorphism between the respective symmetric algebras,  with  gradings $S(V) = \bigoplus_{i\ge 0} S(V)_i, S(U) = \bigoplus_{i\ge 0} S(U)_i$ induced by the gradings on $V$ and $U$,  respectively.
For $i>0$, let $\tau_i = \tau|_{V_i}$ and $d_0 \tau_i : V_i \to U_i$ be the composition of $\tau_i$ with the projection on $U_i$ along $S(U)_i^{>1} = S(U)_i \cap S(U)^{>1}$. Then 
$\tau$ is surjective if and only if its linear term $d_0 \tau= \bigoplus_{i=1}^n d_0 \tau_i : V \to U$ is surjective. 
\end{Lemma}
\begin{proof}
Observe that $\tau_i = d_0 \tau_i + \btau_i : V_i \to S(U)_i$ where $\btau_i : V_i \to S(U)_i^{>1}$ is the composition of $\tau_i$ with the projection on $S(U)_i^{>1} $ along $U_i$.
Suppose $\tau$ surjects, so for $u \in U_i$ there exists $f\in S(V)_i$ such that $\tau(f) = u$. Since $\tau(S(V)^{>1})\subset S(U)^{>1}$ we see that, if $f' = f'_0 + f'_1 + f'_{>1} \in \C \oplus V \oplus S(V)^{>1}$, then $\tau(f') = \tau(f'_1) = u$. It follows that $d_0\tau(f') = u$, hence $d_0 \tau$ surjects. Now suppose that $d_0\tau$ surjects and that $u \in U_i$ with $d_0\tau (v) = u$. Then $\tau_i(v) = u + \btau_i(v)$ and an inductive argument shows that $\btau_i(v)$ lies in the image of $\tau$. Hence $U$ lies in the image, which proves that $\tau$ is surjective.
\end{proof}

\begin{Lemma}\label{L:existenceofphi}
Let $A = \C[\Ss_{\chi,\N}]$. Then,  there is a unique graded algebra morphism $\phi \colon \C[\h^*/{W}] \to \C[\h^*_0/{W_0}]$
such  that $\PD_A(\phi)(\C[\Ss_\chi], \iota) = (\C[\Ss_{\chi_0}],  \iota_0)$ and it is surjective.
\end{Lemma}
\begin{proof}
 By Theorem~\ref{T:LSNtheorem},
 $u := (\C[\Ss_{\chi}], \iota) \in \PD_A(\C[\h^*/{W}])$ is a universal Poisson deformation of $A$.  By Lemmas \ref{L:Walgebragraded} and \ref{L:subregularPoissoniso}  $(\C[\Ss_{\chi_0}], \iota_0)$ is a Poisson deformation of $A$ over $\C[\h_0^*/{W_0}]$, so the universal property for $u$ gives the existence of $\phi$.  By Lemma \ref{L:auxiliary} from which we retain notiation,  surjectivity of $\phi$ follows from surjectivity of $d_0\phi$.   
 By $\C^\times$-semi-universality of $\C[\Ss_{\chi}]$ (see Section 2.5 and Theorem 8.7 of \cite{Slo}) the differentials at zero are equal for the morphism $j$ of \eqref{e:slodowy} and the morphism $\h_0^*/W_0 \to \h^*/W$ whose pull-back is $\phi$. 
  Algebraically this means precisely that $d_0 \phi=d_0 j^*$. 
   The latter is surjective by Lemma \ref{L:auxiliary} because $j$ is a closed inclusion of affine varieties.
\end{proof}

\begin{Theorem}
\label{T:subregularuniversal_specialcase}
Let $\g_0$ be of type {\sf B}$_n$, {\sf C}$_n$ or {\sf F}$_4$, where $n \ge 2$ and $n$ is even in type {\sf C}.  Let $A = \C[\Ss_{\chi,\N}]$.  Then
\begin{enumerate}[label=(\arabic*)]
\item $\PD_{A, \Gamma_0}$ is represented by $\C[\h^*_0/W_0]$ and
$(\C[\Ss_{\chi_0}], \iota_0)$ is a universal element;
\item $\Qnt_{A, \Gamma_0}$ is represented by  $Z(\g_0, e_0)$ and $(U(\g_0, e_0), \iota_0)$ is a universal element.
\end{enumerate}
\end{Theorem}
\begin{proof}
We prove (1); then  (2) follows from Theorem \ref{T:firstmain} (2) and Theorem \ref{T:Walgfilteredquant}.

Let $\alpha_{\Gamma_0} \in \Hom_{\GrAlg}(\C[\h^*/{W}], (\C[\h^*/{W}])_{\Gamma_0})$ be the quotient map to the coinvariant algebra and let $\phi$ be the  morphism in Lemma \ref{L:existenceofphi}.  By  Proposition~\ref{P:dforgetful}, $u^{\Gamma_0}:=\PD_A(\alpha_{\Gamma_0})(u) $ is 
a universal $\Gamma_0$-deformation of $A$,  so we first compare the surjective graded morphisms $\alpha _{\Gamma_0}$ and $\phi$. 
We claim that $\ker(\phi) = \ker(\alpha_{\Gamma_0})$.
The kernel of $\alpha_{\Gamma_0}$ is  generated by $f - f\cdot \gamma$ where $f\in \C[\h^*/W]$ and $\gamma \in \Gamma_0$, however we will obtain a different description of the kernel.  
If $r$ denotes the rank of $\g$ then we write $(d_i)_{i=1}^r$ 
for the Kazhdan graded degrees of the elementary homogeneous generators $e_1,...,e_r$ of $\C[\h^*/W]$.
These degrees are listed in \cite[p. 112]{Slo}, and they coincide with the total degrees doubled, viewed as polynomials on $\h^*$.
Let $\Lambda_0, \Lambda_2 \subseteq \{1,...,r\}$ be the two complementary sets consisting of indexes $i$ such that $d_i$ is congruent to $0$ or $2 \mathrm{ mod } 4$, respectively.
Thanks to our restrictions on the Dynkin label of $\g_0$ the set $\{d_i \mid i \in \Lambda_0\}$ coincides with the collection of all degrees of homogeneous generators of $\C[\h^*_0]^{W_0}$, whilst  $\dim \h_0 = |\Lambda_0|$.
Therefore the Kazhdan grading on $\C[\h_0^*/W_0]$ has degree concentrated in $4\Z$.
Since $\phi$ is graded, the generators of degree $d_i$ with $i \in \Lambda_2$ are mapped to zero.
Since $\phi$ is surjective, the generators with degrees $d_i$ where $i\in \Lambda_0$ are sent to algebraically independent elements.
It follows that $\ker \phi  = (e_i \mid i \in \Lambda_2)$.

It is explained in \cite[\textsection 13]{Ca} (see also \cite[Remark 8.8.4]{Slo}) that $\C[\h^*/W]_{\Gamma_0} \simeq \C[\h_0^*/W_0]$ as algebras graded by total degree and, equivalently, by Kazhdan degree.
Since $\alpha_{\Gamma_0}$ is a surjection we can apply the argument of the previous paragraph replacing $\phi$ with $\alpha_{\Gamma_0}$ to deduce that $\ker \alpha_{\Gamma_0} = (e_i \mid i \in \Lambda_2)$.

Equality of the kernels gives the existence of a graded isomorphism $\sigma \colon \C[\h_0^*/W_0] \isoto \C[\h^*/W]_{\Gamma_0}$ by setting $\sigma(\phi(f)) := \alpha_{\Gamma_0}(f)$ for $f\in \C[\h^*/W]$.
We have a commutative triangle of graded homomorphisms:
\begin{eqnarray}
\label{e:littletriangle}
\begin{tikzcd}
 & \C[\h^*/W] \arrow{dl}[swap]{\phi} \arrow{dr}{\alpha_{\Gamma_0}} \\
\C[\h_0^*/W_0] \arrow{rr}{\sigma}  && \C[\h^*/W]_{\Gamma_0}.
\end{tikzcd}
\end{eqnarray}
The isomorphism $\sigma$ is $\Gamma_0$ invariant and it satisfies
$\PD_{A, \Gamma_0}(\sigma)\PD_A(\phi)u = u^{\Gamma_0}$.
This proves that $\C[\h_0^*/W_0]$ is  another choice for a universal base of $\PD_{A, \Gamma_0}$  and that $\PD_A(\phi)u$ is a universal element of $\PD_{A, \Gamma_0}$ over this base.
\end{proof}

We conjecture that Theorem~\ref{T:subregularuniversal_specialcase} holds in general, without the restrictions on Dynkin type. We have the following consequence. 

\begin{Corollary}\label{C:subregularuniversal_specialcase}
There is a surjective homomorphism $U(\g,e) \twoheadrightarrow U(\g_0, e_0)$.
Under the assumptions of Theorem~\ref{T:subregularuniversal_specialcase} the kernel is generated by $\{ z - z \cdot \gamma \mid \gamma\in \Gamma_0, z\in Z(\g,e)\}$. 
\end{Corollary}
\begin{proof} 
Retain the notation $A = \C[\Ss_{\chi,\N}]$.
Let $q_0 := (U(\g_0,e_0),  \iota_0) \in \Qnt_A (Z(\g_0, e_0))$.
By Theorem~\ref{T:Walgfilteredquant} we see that $q := (U(\g,e),\iota) \in \Qnt_A( Z(\g, e))$ is a universal filtered quantization and $\gr q$ a universal Poisson deformation of $A$, 
so there is a unique morphism $\beta \colon Z(\g, e) \to Z(\g_0, e_0)$ in $\SFAlg$ such that $\Qnt_A(\beta)(q) = q_0$.
In particular, there exists a filtered $Z(\g_0, e_0)$-linear isomorphism $U(\g,e) \otimes_{Z(\g, e)} Z(\g_0, e_0) \to U(\g_0,e_0)$.
 Moroever, 
$\PD_A(\gr \beta)(\gr q) = \gr q_0$, by Theorem \ref{T:firstmain}.
Lemma \ref{L:existenceofphi} gives surjectivity of  $\gr \beta$, whence of  $\beta$.
The sought map is then the composition of morphisms:
$$U(\g,e) \to U(\g,e) \otimes_{Z(\g, e)} Z(\g_0, e_0) \to U(\g_0,e_0),$$
where the first arrow is the map $x\mapsto x \otimes 1$ for all $x \in U(\g,e)$.
It is surjective by surjectivity of $\beta$.
The statement regarding the kernel follows directly from the proof of Theorem \ref{T:subregularuniversal_specialcase}.
\end{proof}

\subsection{A presentation for the subregular $W$-algebra of type {\sf B}}

In this section we let $G_0 = \SO_{2n+1}$ and $\g_0 = \Lie(G_0)$. Let $e_0 \in \g_0$ be a subregular nilpotent element of $\g_0$ and $\chi_0 \in \g_0^*$ the corresponding element with respect to the Killing identification. Our purpose here is to give a presentation of the finite $W$-algebra $U(\g_0, e_0)$ as a quotient of a shifted Yangian.

By Corollary~\ref{C:subregularuniversal_specialcase} we can express $U(\so_{2n+1}, e)$ as a quotient of $U(\sl_{2n}, e)$, whilst \cite{BKshift} allows us to express $U(\gl_{2n}, e)$ as a truncated shifted Yangian. In order to tie these threads together we record the following observation which follows straight from the definitions.
\begin{Lemma}
\label{L:glnvsln}
The centre of $\gl_n$ maps to a $1$-dimensional central subspace of $U(\gl_n, e)$ and the quotient by that subspace is isomorphic to $U(\sl_n, e)$. \hfill \qed
\end{Lemma}

In \cite{BKshift} the {\it shifted Yangian associated to $\gl_n$} is introduced in full generality, however in this paper we only require a special case: we define the shifted Yangian $Y_2(\sigma)$ to be the algebra with (infinitely many) generators
\begin{eqnarray}
\begin{array}{c}
\{D_1^{(r)}, D_2^{(r)} \mid r>0\} \cup \{E^{(r)} \mid r> 2n-2 \} \cup \{F^{(r)} \mid r> 0\}
\end{array}
\end{eqnarray}
and relations (2.4)-(2.9) from \cite{BKshift}. Our generators $E^{(r)}$ and $F^{(r)}$ are denoted $E_1^{(r)}$ and $F_1^{(r)}$ in {\it loc. cit.} and our definition above corresponds to the shift matrix $\sigma = (s_{i,j})_{1\leq i,j\leq 2}$ with $s_{1,2} = 2n-2$ and $s_{i,j} = 0$ otherwise. We gather the diagonal generators $D_i^{(r)}$ into power series by setting $D_i(u) := \sum_{r \ge 0} D_i^{(r)} u^{-r} \in Y_2(\sigma)[[u^{-1}]]$ where $D_i^{(0)} := 1$ and consider the series
\begin{eqnarray}
\label{e:Zdefined}
\begin{array}{rcl}
Z(u) &=& u^{2n} + \displaystyle \sum_{r > 0} Z^{(r)} u^{2n-r} \vspace{4pt} \\ 
&:=& u (u-1)^{2n-1} D_1(u) D_2(u-1) \in u^{2n} Y_2(\sigma)[[u^{-1}]].
\end{array}
\end{eqnarray}

\begin{Lemma}
The elements $\{Z^{(r)} \mid r> 0\}$ are algebraically independent generators of the centre of $Y_2(\sigma)$. Furthermore for $r = 1,..., 2n$ we have
\begin{eqnarray}
\label{e:centreYangian}
Z^{(r)} = \sum_{s=0}^r \binom{2n-1}{2n-1-s} (-1)^{2n-s} \sum_{t=0}^s D_1^{(t)} \cD_2{}^{(s-t)}
\end{eqnarray}
where $\cD_2{}^{(r)} := \sum_{s=0}^r \binom{r-1}{r-s} D_2^{(s)}$ and $\cD_2{}^{(-1)} := 0$.
\end{Lemma}
\begin{proof}
The first claim follows from \cite[Theorem~2.6]{BKrep} in view of the fact that $u^{-2n+1} (u-1)^{2n-1}$ is invertible in $\C[[u^{-1}]]$. We proceed to prove formula \eqref{e:centreYangian}. Using the binomial theorem we have $(u-1)^{-s} = \sum_{r\ge s} \binom{r-1}{r-s} u^{-r}$. 
It follows that
\begin{eqnarray}
D_2(u-1) = \sum_{r\ge 0} (u-1)^{-r} D_2^{(r)} = \sum_{r\ge 0} u^{-r} \sum_{s=0}^r D_2^{(s)} \binom{r-1}{r-s} = \sum_{r\ge 0} u^{-r} \cD_2{}^{(r)}.
\end{eqnarray}
If we define $C(u) = \sum_{r\ge 0} C^{(r)}u^{-r} := D_1(u) D_2(u-1)$ then we have
\begin{eqnarray}
\label{e:Cformula}
C(u) = \sum_{r, s\ge 0} D_1^{(r)} \cD_2{}^{(s)} u^{-r-s} = \sum_{r \ge 0} \sum_{s=0}^r u^{-r} D_1^{(s)} \cD_2{}^{(r-s)}.
\end{eqnarray}
At the same time we have 
\begin{eqnarray}
\label{e:polyformula}
u(u-1)^{2n-1} = \sum_{i=1}^{2n} \binom{2n-1}{i-1} (-u)^i.
\end{eqnarray}
Finally if we have a polynomial $f(u) = \sum_{i=0}^m f_i u^i$ and a power series $A(u) = \sum_{r\ge 0} A_r u^{-r}$ then for $r=0,...,m$ the $u^{m-r}$ coefficient of $f(u) A(u)$ is
$\sum_{s=0}^r f_{m-s} A_s$. Since $\binom{2n-1}{-1} = 0$ we can combine this last statement together with \eqref{e:Cformula} and \eqref{e:polyformula} we arrive at the proof of \eqref{e:centreYangian}.
\end{proof}

\begin{Theorem}
\label{T:Yangianstheorem}
There is a surjective algebra homomorphism
$$Y_2(\sigma) \longtwoheadrightarrow U(\g_0, e_0)$$
with kernel generated by
$$\{D_1^{(r)} \mid r > 1\} \cup \{ Z^{(2r - 1)} \mid r=1,...,n\}.$$
\end{Theorem}
\begin{proof}
Let $e$ be a subregular nilpotent element of $\sl_{2n}\subseteq \gl_{2n}$. The main result of \cite{BKshift} implies that there is a surjective homomorphism $Y_2(\sigma) \onto U(\gl_{2n}, e)$ with kernel generated by $\{D_1^{(r)} \mid r > 1\}$. It follows from \cite[Lemma~3.7]{BKrep} that the image of the element $Z^{(1)}$ in \eqref{e:Zdefined} under the map $Y_2(\sigma) \to U(\gl_{2n}, e)$ lies in the image of $\z(\gl_{2n}) \to Z(\gl_{2n}) \to U(\gl_{2n}, e)$. Together with Lemma~\ref{L:glnvsln} this implies that $U(\sl_{2n}, e)$ is naturally isomorphic to the quotient of $U(\gl_{2n}, e)$ by $Z^{(1)}$. Finally by Example~\ref{E:ABexample} and Corollary~\ref{C:subregularuniversal_specialcase} there is a surjective algebra homomorphism $U(\sl_{2n},e) \onto U(\g_0, e_0)$ and the kernel is generated by the image of the elementary symmetric polynomials $\{e_{2r+1}\mid r = 1,...,n-1\}$ under the isomorphism $\C[\h^*/W] \to Z(\sl_{2n},e)$ discussed in \eqref{e:centres}. Here we use $(\h, W)$ to denote a torus and Weyl group for $\sl_{2n}$. To complete the proof of the current Theorem it suffices to show, for $r=1,...,n-1$, that the image of $e_{2r+1}$ under $\C[\h^*/W] \to Z(\sl_{2n},e)$ is equal to the image of $Z^{(2r+1)}$ under $Y_2(\sigma) \to U(\gl_{2n},e) \to U(\sl_{2n},e)$. Once again this follows from \cite[Lemma~3.7]{BKrep}.
\end{proof}

\vspace{10pt}

\noindent {Contact details:}\vspace{5pt}

Filippo Ambrosio {\sf filippo.ambrosio@uni-jena.de};\\
Friedrich-Schiller-Universit{\"a}t Jena FMI, Ernst-Abbe-Platz 2, 07743 Jena, Germany.\vspace{5pt}

Giovanna Carnovale {\sf carnoval$@$math.unipd.it};

Francesco Esposito {\sf esposito$@$math.unipd.it};\\
Dipartimento di Matematica ``Tullio Levi-Civita" (DM), Via Trieste, 63 - 35121 Padova, Italy.\vspace{5pt}

Lewis Topley {\sf lt803$@$bath.ac.uk};\\
Department of Mathematical Sciences, University of Bath, Claverton Down, Bath, BA2 7AY, UK.


\begin{thebibliography}{999}


\bibitem{Art} {\sc M. Artin},
On isolated rational singularities of surfaces.
{\it Amer. J. Math.} {\bf 88} \ (1966), 129--136.

\bibitem{Be}{\sc A. Beauville},
Symplectic singularities.
{\it Invent. Math.} {\bf 139} \ (2000), no. 3, 541--549.

\bibitem{Bel}{\sc G. Bellamy},
Counting resolutions of symplectic quotient singularities.
{\it Compos. Math.} {\bf 152} \ (2016), no. 1,  99--114.

\bibitem{Bel23} \bysame,
Coulomb Branches have symplectic singularities.
{\tt 	arXiv:2304.09213} (2023).





\bibitem{BFN} {\sc A. Braverman, M. Finkelberg, H. Nakajima},
Towards a mathematical definition of Coulomb branches of 3-dimensional $\N = 4$ gauge theories, II.
{\it Adv. Theor. Math. Phys.} {\bf 22} \ (2019), no. 5, 1071--1147. 

\bibitem{Br}{\sc E. Brieskorn},
``Singular elements of semi-simple algebraic groups''.
Actes du Congr{\`e}s International des Math{\'e}maticiens (Nice, 1970), Tome 2, pp. 279--284. Gauthier-Villars, Paris, 1971.

\bibitem{BKshift} {\sc J. Brundan, A. Kleshchev},
Shifted Yangians and finite $W$-algebras.
{\it Adv.\ Math.} \ {\bf 200} (2006), 136--195.

\bibitem{BKrep}
\bysame,
``Representations of shifted Yangians and finite $W$-algebras''.
Mem.\ Amer.\ Math.\ Soc.\ {\bf 196} (2008).


\bibitem{Ca}
{\sc R. W. Carter},
``Simple groups of Lie type''.
A Wiley-Interscience Publication, London, 1972.

\bibitem{CM} {\sc J-Y. Charbonnel \& A. Moreau},
The symmetric invariants of centralisers and Slodowy grading.
{\it Math. Z.} {\bf 282} \ (2016), no. 1-2, 273--339.
 
\bibitem{DKV} {\sc A. De Sole, V. G. Kac, D. Valeri},
Structure of classical (finite and affine) $\mathscr{W}$-algebras.
{\it J. Eur. Math. Soc. (JEMS)} {\bf 18} \, (2016), no. 9, 1873--1908.

\bibitem{EG}
{\sc P. Etingof, V. Ginzburg},
Symplectic reflection algebras, Calogero--Moser space, and deformed Harish-Chandra homomorphism.
{\it Invent. Math.} {\bf 147} \ (2002), no. 2, 243--348.
 
\bibitem{GG}
{\sc W.~L.~Gan, V.~Ginzburg},
Quantization of Slodowy slices.
{\it Internat. Math. Res. Notices} {\bf 5} \ (2002), 243--255.

\bibitem{Gi} {\sc V. Ginzburg},
Harish-Chandra bimodules for quantized Slodowy slices.
{\it Represent. Theory} {\bf 13} \ (2009), 236--271.

\bibitem{Hoch} {\sc G. P. Hochschild},
``Basic  theory of  algebraic  groups and  Lie  algebras''.
 GTM, vol. 75, Springer-Verlag,  1981.

\bibitem{JaNO}
{\sc J.~C.~Jantzen},
``Nilpotent orbits in representation theory''.
Progress in Math., vol.\ 228, Birkh\"auser, 2004.

\bibitem{Ka} {\sc D. Kaledin},
Symplectic singularities from the Poisson point of view.
{\it J. Reine Angew. Math.} {\bf 600} \ (2006), 135--156.

\bibitem{Ko1} {\sc B. Kostant},
Lie group representations on polynomial rings.
{\it Amer. J. Math.} {\bf 85} \ (1963), 327--404.

\bibitem{Ko2}\bysame,
On Whittaker vectors and representation theory.
{\it Invent. Math.} {\bf 48} \ (1978) no. 2, pp. 101--184.

\bibitem{LPV}{\sc C. Laurent-Gengoux, A. Pichereau, P. Vanhaecke},
``Poisson structures''.
Grundlehren der Mathematischen Wissenschaften [Fundamental Principles of Mathematical Sciences], 347. Springer, Heidelberg, 2013.

\bibitem{LNS} {\sc M. Lehn, Y. Namikawa, Ch. Sorger},
Slodowy slices and universal Poisson deformations.
{\it Compos. Math.} {\bf 148}\ (2012), no. 1, pp. 121--144.

\bibitem{Lo12} {\sc I. Losev},
Isomorphisms of quantizations via quantization of resolutions.
{\it Adv. Math.} {\bf 231} \ (2012), no. 3-4, 1216--1270.

\bibitem{Lo} \bysame,
Deformations of symplectic singularities and Orbit method for semisimple Lie algebras.
{\it Sel. Math. New Ser.} {\bf 28}, 30, 2022.


\bibitem{MacLane}{\sc S. MacLane},
``Categories for the Working Mathematician''.  Graduate Texts in Mathematics,  5.  Springer-Verlag,  New York, 1971.

\bibitem{MR} {\sc J. C. McConnell, J. C. Robson},
``Non-commutative Noetherian rings''.
Revised edition. Graduate Studies in Mathematics, 30. American Mathematical Society, Providence, RI, 2001.

\bibitem{MN}{\sc K. McGerty, T. Nevins},
Kirwan surjectivity for quiver varieties.
{\it Invent. Math.} {\bf 212} \ (2018), no. 1, 161--187.

\bibitem{W2}{\sc D. Muthiah, A. Weekes},
Symplectic leaves for generalized affine Grassmannian slices. {\tt arXiv:1902.09771}  (2019).

\bibitem{Na1} {\sc Y. Namikawa},
Poisson deformations of affine symplectic varieties.
{\it Duke Math. J.} {\bf 156} \ (2011), no. 1, 51--85.

\bibitem{Na2} \bysame, 
Poisson deformations of affine symplectic varieties, II. 
{\it Kyoto J. Math.} {\bf 50} \ (2010), no. 4, 727--752.

\bibitem{NVO}{\sc C. N\v{a}st\v{a}sescu, F. Van Oystaeyen},
``Graded and Filtered Rings and Modules'', LNM 758, Springer-Verlag, 1979.

\bibitem{PrST}
{\sc A. Premet},
Special transverse slices and their enveloping algebras.
{\it Adv.\ Math.} {\bf 170} \ (2002), 1--55.

\bibitem{PrJI}\bysame,
Enveloping algebras of Slodowy slices and the Joseph ideal.
{\it J. Eur. Math. Soc. (JEMS)} {\bf 9} \ (2007), no. 3, 487--543.

\bibitem{PrMF}
\bysame,
Multiplicity-free primitive ideals associated with rigid nilpotent orbits.
{\it Transform. Groups} {\bf 19} \ (2014), no. 2, 569--641.



\bibitem{Slo}
{\sc P. Slodowy},
``Simple Singularities and Simple Algebraic Groups''.
Lecture Notes in Mathematics 815, Springer-Verlag, Berlin Heidelberg, 1980.

\bibitem{W1}
{\sc A. Weekes}, Quiver gauge theories and symplectic singularities. 
{\it Adv. Math.} {\bf 396} (2022), Paper No. 108185, 21 pp.


\end{thebibliography}
\end{document}